\documentclass[12pt,a4paper]{article}

\usepackage{dsfont}
\usepackage{amssymb}
\usepackage{amsthm}
\usepackage{amsfonts}
\usepackage{amsmath}
\usepackage{dsfont}
\usepackage{anysize}
\usepackage{bbm}
\usepackage{setspace}
\usepackage{color}
\usepackage[top=1in, bottom=1in, left=1in, right=1in]{geometry}
\usepackage[style=trad-abbrv,maxnames=99,maxalphanames=9]{biblatex}
\usepackage{hyperref}
\usepackage[hyphenbreaks]{breakurl}

\usepackage[nottoc]{tocbibind}

\AtBeginBibliography{\footnotesize}
 
\renewbibmacro{in:}{}

\providecommand{\keywords}[1]
{	
  \textbf{\textit{Keywords:}} #1
}

\DeclareFieldFormat[article]{citetitle}{#1}
\DeclareFieldFormat[article]{title}{#1}
\DeclareFieldFormat[inbook]{citetitle}{#1}
\DeclareFieldFormat[inbook]{title}{#1}
\DeclareFieldFormat[incollection]{citetitle}{#1}
\DeclareFieldFormat[incollection]{title}{#1}
\DeclareFieldFormat[inproceedings]{citetitle}{#1}
\DeclareFieldFormat[inproceedings]{title}{#1}
\DeclareFieldFormat[phdthesis]{citetitle}{#1}
\DeclareFieldFormat[phdthesis]{title}{#1}
\DeclareFieldFormat[misc]{citetitle}{#1}
\DeclareFieldFormat[misc]{title}{#1}
\DeclareFieldFormat[book]{citetitle}{#1}
\DeclareFieldFormat[book]{title}{#1} 

\newtheorem*{theorem*}{Theorem}

\newtheorem{theorem}{Theorem}[section]

\newtheorem{lemma}[theorem]{Lemma}
\newtheorem{corollary}[theorem]{Corollary}

\newtheorem{remark}[theorem]{Remark}

\newtheorem{definition}[theorem]{Definition}
\newtheorem{proposition}[theorem]{Proposition}

\def\del{\partial}
\def\dbar{\bar\partial}
\def\ddbar{\del\dbar}

\newcommand{\RR}{\mathbb{R}}

\def\del{\partial}

\DeclareMathOperator{\Ric}{Ric}

\def\NN{\mathbb{N}}
\def\RR{\mathbb{R}}
\def\ord{\mathrm{ord}}

\newcommand{\PSH}{\mathrm{PSH}}
\usepackage{mathtools}

\DeclareMathOperator{\vol}{vol}
\DeclareMathOperator{\usc}{usc}

\newcommand{\ddc}{\mathrm{dd}^{\mathrm{c}}}

\def\beq{\begin{equation}}
\def\eeq{\end{equation}}

\bibliography{ComplexGeometry}

\title{Twisted K\"ahler--Einstein metrics in big classes}
\author{Tam\'as Darvas (University of Maryland) \\ Kewei Zhang (Beijing Normal University) }
\date{\vspace{-0.5cm}}

\begin{document}
\maketitle

\begin{abstract}We prove existence of twisted  K\"ahler--Einstein metrics in big cohomology classes, using a divisorial stability condition. In particular, when $-K_X$ is big, we obtain a uniform Yau--Tian--Donaldson existence theorem for K\"ahler--Einstein metrics. 
To achieve this, we build up from scratch the theory of  Fujita--Odaka type delta invariants in the transcendental big setting, using pluripotential theory. We do not use the K-energy in our arguments,  and our techniques provide a simple roadmap to prove Yau--Tian--Donaldson existence theorems for K\"ahler--Einstein type metrics, that only needs  convexity of the appropriate Ding energy. As an application, we give a simplified proof of Li--Tian--Wang's existence theorem  in the log Fano setting.
\end{abstract}

\keywords{K\"ahler--Einstein metric,  geodesic ray, Yau--Tian--Donaldson theorem.}

%\setcounter{tocdepth}{1}
%\tableofcontents

\section{Introduction}

The study of canonical metrics in K\"ahler geometry goes back to early work of Calabi \cite{Cal57} and Yau's solution of the Calabi conjecture \cite{Y78}. The different versions of the Yau--Tian--Donaldson (YTD) conjecture predict that existence of canonical metrics is equivalent to various algebraic stability conditions of the underlying manifold. We refer to the textbooks \cite{Szebook,Tibook} for an introduction to this exciting topic.

Among canonical metrics, K\"ahler--Einstein (KE) type metrics are of particular interest, whose study revealed numerous deep connections between differential geometry and algebraic geometry.
Here we 
develop a transcendental pluripotential theoretic approach to  KE metrics, which not only gives simplified proofs of some existing YTD type theorems in the literature but also allows one to treat KE metrics even in transcendental big cohomology classes.  More specifically, we will
study divisorial stability and geodesic stability of big classes, and obtain a uniform Yau--Tian--Donaldson type existence theorem for twisted K\"ahler--Einstein metrics in this general setting.

We recall the definition of the Ricci curvature and KE type metrics in the big context. Let $(X,\omega)$ be a compact K\"ahler manifold, and $\theta$ a smooth closed $(1,1)$-form representing a big cohomology class $\{\theta\}$ in $H^{1,1}(X,\RR)$. All volumes and measures in this work are interpreted in the non-pluripolar context of \cite{BEGZ10} (see Section \ref{sec:psh-theory}). 

Given $u \in \mathcal E^1(X,\theta)$ (see Definition \ref{def: fullmass}) suppose that the non-pluripolar measure $\theta_u^n:= (\theta +\ddc u)^n$ is absolutely continuous with respect to $\omega^n$. If $ \log \big(\theta_u^n /\omega^n \big)$ is integrable, then we introduce $\Ric \theta_u$ as the following current, with the convention $\ddc:=i\partial\bar\partial/2\pi$:
\begin{equation}\label{eq: Ric_def}
\Ric \theta_u = \Ric \omega  - \ddc \log \frac{\theta_u^n}{\omega^n}.
\end{equation}

Now we define twisted K\"ahler--Einstein currents/metrics. 
Let $\eta$ be a smooth closed $(1,1)$-form representing the cohomology class for which the following decomposition holds:
\begin{equation*}\label{eq: c_1_decomp}
c_1(-K_X) = \{\theta\} + \{\eta\}.
\end{equation*}
Hodge theory provides $f \in C^\infty(X)$ such that $\Ric \omega = \theta + \eta + \ddc f$. Let $\psi$ be a quasi-plurisubharmonic function on $X$ and $\eta_\psi := \eta + \ddc \psi$ will be our twisting. We want to find $u \in \textup{PSH}(X,\theta)$, with minimal singularity type, satisfying the $\eta_\psi$-twisted KE equation:
\begin{equation*}\label{eq: KE_current}
\Ric  \theta_u = \theta_u + \eta_\psi.
\end{equation*}

Using the identity $\Ric \omega = \theta + \eta + \ddc f$ and  \eqref{eq: Ric_def} we reduce the twisted KE equation to the following highly degenerate complex Monge--Amp\`ere equation:
\begin{equation}\label{eq: KE_scalar equation}
(\theta + \ddc u)^n  = e^{-u +f - \psi} \omega^n.
\end{equation}

%Immediately we see that $e^{-\psi}$ needs to be integrable for the right hand side to make sense, i.e., $\psi$ needs to have klt singularity.

Already in the K\"ahler case it is well known that the above equation is not always solvable, with a whole host of algebraic criteria available, under the umbrella of K-stability. 

In the Fano case, when $\{\theta\}  = c_1(-K_X)$,  YTD type existence theorems were proved using the continuity method/Cheeger--Colding--Tian theory \cite{CDS15, DS16a,  Tian15, TiWa20}, K\"ahler--Ricci flow \cite{CSW} and the variational/non-Archimedean method \cite{BBJ21}. Recently the second author gave a proof using K\"ahler quantization  \cite{Zh21}. 

In the singular setting of log-Fano pairs, YTD type existence theorems were proved in \cite{Li20YTD,LTW21a,LTW21b,LXZ22}.  
After a resolution of singularities, this case is roughly equivalent with looking for twisted KE metrics in big and semi-positive classes.

Regarding greater generality, we circumvent the difficulties that typically arise in the big setting by developing an analytic approach to divisorial stability via pluripotential theory. Of course, the original definition of K-stability going back to Tian \cite{Ti97} and extended by Donaldson \cite{Don02} is not applicable, as our big class $\{\theta\}$ is not necessarily induced by a line bundle. However the more recent interpretation of K-stability in terms of divisorial data, going back to Fujita \cite{Fuj19} and Li \cite{Li17}, adapts naturally to our context. Indeed, based on the Blum--Jonsson interpretation of the Fujita--Odaka delta invariant \cite{BJ17, BoJ18,FO18}, rooted in the non-Archimedean approach to K-stability (see \cite[Definition 7.2]{BBJ21}), we define the twisted delta invariant of our data:
\begin{equation}\label{eq: delta_psi_def}
\delta_\psi(\{\theta\}):=\inf_E\frac{A_\psi(E)}{S_\theta(E)}.
\end{equation}
Here $
A_\psi(E):=A_X(E)-\nu(\psi,E)$ and the infimum is taken over all prime divisors $E$ over $X$, i.e., $E$ is a prime divisor inside a K\"ahler manifold $Y$ with $\pi: Y \to X$ being a proper bimeromorphic map. Recall that a prime divisor inside $Y$ is an irreducible analytic set of codimension $1$.

When $X$ is projective, one can equivalently only consider projective birational morphisms $\pi$, as  explained in \S\ref{sec:log-Fano}.
Here the log discrepancy $A_X(E)$ of $E$ is
$
A_X(E):=1+\mathrm{coeff}_E(K_Y-\pi^*K_X),
$
and $\nu(\psi,E)$ denotes the Lelong number of $\pi^*\psi$ at a very generic point of $E$ (see \eqref{eq:def-Lelong-number}).
The \emph{expected Lelong number} $S_\theta(E)$ of $\{\theta\}$ along $E$ is defined by
\begin{equation*}\label{eq: SE_def}
S_\theta(E):=\frac{1}{\vol(\{\theta\})}\int_0^{\tau_\theta(E)}\vol(\{\pi^*\theta\}-x\{E\})\mathrm dx,
\end{equation*}
where $\tau_\theta(E) := \sup\{\tau \in \RR \textup{ s.t. } \{\pi^*\theta\}-\tau\{E\} \textup{ is big}\}$ is the pseudoeffective threshold. We emphasize that the volume function $\vol(\cdot)$ is understood in the sense of \cite{BEGZ10}, but it coincides with the algebraic volume in case $\{\theta\}$ is in the N\'eron--Severi space \cite[Proposition 1.18]{BEGZ10}. Hence our delta invariant extends the one by Fujita--Odaka to the transcendental case. When $\eta=0$ and $\psi =0$, we use the much simpler notation $\delta: = \delta_0$.

\begin{definition}
We say $(X,\{\theta\},\eta_\psi)$ is $\eta_\psi$-twisted uniformly K-stable if $\delta_\psi(\{\theta\})>1$.
\end{definition}

Some remarks are in order, to motivate this definition.
In the Fano case, it is known by \cite[Theorem B]{BJ17} that $\delta(-K_X) > 1$ if and only if $(X,-K_X)$ is uniformly K-stable/Ding stable (as introduced in \cite{BHJ17,De16}/\cite{Ber16,BBJ21}). Recent work of Liu--Xu--Zhuang \cite{LXZ22} shows that $\delta(-K_X)>1$ if and only if $(X,-K_X)$ is K-stable.
This condition is further equivalent with \eqref{eq: KE_scalar equation} having a unique solution  \cite[Theorem A]{BBJ21}. The main result of this paper is establishing the historically ``harder" direction of this equivalence in the big context, in the presence of positive twisting current:

\begin{theorem}\label{thm: main_KE_twist_exist}
Suppose that $\eta_\psi \geq 0$. If $\delta_\psi(\{\theta\}) > 1$ then \eqref{eq: KE_scalar equation} has a solution $u \in \textup{PSH}(X,\theta)$ with minimal singularity, i.e., $\Ric \theta_u = \theta_u + \eta_\psi$.
\end{theorem}

We note the following corollary, in the particular case when  $c_1(-K_X) = \{\theta\}$ and $\psi =0$:

\begin{corollary}\label{cor: main_KE_twist_exist}
Suppose that $-K_X$ is big. If $\delta(-K_X) > 1$ then \eqref{eq: KE_scalar equation} has a solution $u \in \textup{PSH}(X,\theta)$ with minimal singularity, i.e., $\theta_u$ is a KE metric satisfying $\Ric \theta_u = \theta_u$.
\end{corollary}

When $X$ is Fano, the argument of Corollary \ref{cor: main_KE_twist_exist} gives a novel pluripotential theoretic proof of the classical (uniform) YTD conjecture, in addition to the ones mentioned above. In fact, as we shall see in \S \ref{sec:ample-model}, the KE metrics found in Corollary \ref{cor: main_KE_twist_exist} are actually in one-to-one correspondence to the KE metrics on the `ample model' of $(X,-K_X)$.

One can produce examples with $-K_X$ big and $\delta(-K_X)>1$. Indeed, let $V$ be a del Pezzo surface with klt singularities and let $X\xrightarrow{\pi}V$ be the minimal resolution of $V$. Then $-K_X=-\pi^*K_V+\sum a_i E_i$, where $a_i\geq0$, $E_i$'s are $\pi$-exceptional curves, and $(E_i\cdot E_j)$ is negative definite. Note that $-K_X$ is big and $-K_X=-\pi^*K_V+\sum a_i E_i$ is the Zariski decomposition of $-K_X$. Also, $-K_X$ is not nef when the singularities of $V$ are worse than Du Val.
Now if $\delta(-K_V)>1$, then one can show that $\delta(-K_X)>1$ as well. Examples of such $V$  are constructed in  \cite{CPS21}. Hence Corollary \ref{cor: main_KE_twist_exist} is applicable in this case. 

We expect that the twisted KE metric found in Theorem  \ref{thm: main_KE_twist_exist} is unique, an open question going back to Berndtsson \cite[page 4]{Bern15}.  
Conversely, if a unique twisted KE metric exists, one can show that $\delta_\psi(-K_X) > 1$ (see Proposition \ref{prop:unique-imply-delta>1}). A similar result has been independently obtained by \cite{DeRe22}, who show that  if a unique twisted KE metric exists then $(X,-K_X)$ is uniformly Ding stable in terms of test configurations.
%Even though a full solution to  Berndtsson's uniqueness question is missing, we do have uniqueness in the setting of Corollary \ref{cor: main_KE_twist_exist}, as it will follow from the discussion in \S \ref{sec:ample-model}; see Theorem \ref{thm:unique-KE}.

% \begin{theorem}
% Suppose that $\psi$ has analytic singularities and that $\eta_\psi \geq 0$. Then \eqref{eq: KE_scalar equation} has a unique solution $u \in \textup{PSH}(X,\theta)$ with minimal singularity if and only if $\delta_\psi(\{\theta\}) > 1$.    
% \end{theorem}

We believe the novelty of our work partly lies in the adaptability and robust nature of our transcendental techniques. Whenever one can prove the convexity of Ding functionals of certain canonical KE metrics, our methods yield a YTD type existence theorem in terms of the appropriate delta invariant. To provide evidence,  we give a much simplified proof of the main result of Li--Tian--Wang \cite{LTW21b} using our techniques. 
Remarkably, after the first version of our work appeared on the arXiv, in  \cite{Xu22}  C. Xu  noticed that by using results of Birkar--Cascini--Hacon--McKernan \cite{BCHM10} one can reduce uniform K-stability of $(X,-K_X)$ to uniform K-stability of a $\mathbb{Q}$-Fano ample model, which gives an algebraic proof of our Corollary 1.3, in addition to the transcendental one given in this work. In Section \ref{sec:log-Fano} we discuss the above connections in detail.

Compared to the seminal work of Berman--Boucksom--Jonsson \cite{BBJ21}, we do not require the convexity of K-energy functionals, nor do we need to approximate geodesic rays via test configurations. This is essential to our approach, as the K-energy functional is known to be convex only in the big and nef case \cite{DNL22}. In the general big case, even a suitable definition of K-energy is still missing to the authors' knowledge.

Our main ingredients  are the valuative criteria for integrability \cite{BBJ21,BFJ08}, (see \cite{Bo17} for an excellent learning source), 
the Guan--Zhou openness theorem \cite{GuZh15}, the relative pluripotential theory developed in \cite{DDNL4,DDNL2,DDNL5}, and the Ross--Witt Nystr\"om correspondence between (sub)geodesic rays and test curves  introduced in \cite{RWN14}, and studied further in \cite{DDNL3,DX20}.

That $\eta_\psi \geq 0$ guarantees that the twisted Ding functional is convex, as follows from the main result of \cite{BePa08}. This is used in our proof of  Theorem \ref{thm: main_KE_twist_exist} in one specific point, but the rest of the argument works in much more general context, allowing to treat many different problems  in the literature at the same time, as we now point out. 

To start, we consider a general quasi-plurisubharmonic (qpsh) function  $\psi$ on $X$ that does not necessarily satisfy  $\eta_\psi \geq 0$. In addition, let $\chi$ be another qpsh function on $X$ with analytic singularity type, and consider the following more general equation of twisted KE type, for $u \in \textup{PSH}(X,\theta)$ with minimal singularity type:
\begin{equation}\label{eq: KE_scalar equation_general}
(\theta + \ddc u)^n  = e^{-u +\chi - \psi} \omega^n.
\end{equation}
To make sense of this equation, we must assume that  $\int_Xe^{\chi-\psi}\omega^n<\infty$, i.e., $\chi -\psi$ is klt.
Clearly this is more general than \eqref{eq: KE_scalar equation}, as any $f \in C^\infty(X)$ is a qpsh function. When treating canonical metrics, it is natural to consider some type of continuity method. This is the case here as well, as we consider the following family of equations, for $\lambda > 0$:
\begin{equation}\label{eq: KE_cont equation}
(\theta + \ddc u)^n  = e^{-\lambda u + \chi - \psi} \omega^n.
\end{equation}
At least for small $\lambda >0$ the Guan--Zhou openness theorem \cite{GuZh15} gives $\int_Xe^{-\lambda v + \chi - \psi}\omega^n<\infty$ for all $v \in \mathcal E^1(X,\theta)$ (See Theorem \ref{thm: openness} and Proposition \ref{prop:c-u=c-P-u} below).

The expression on the right hand side of \eqref{eq: KE_cont equation} often appears in the literature and, after adding a constant, it is convenient to introduce the following Radon probability measure:
\begin{equation}
    \label{eq:def-mu}
    \mu := e^{\chi - \psi} \omega^n.
\end{equation}
As in \cite{BBEGZ16}, such measures $\mu$ are called \emph{tame}. Following \cite{Ding88}, one defines a functional, whose Euler-Lagrange equation is exactly \eqref{eq: KE_cont equation}. This is $\mathcal D_{\mu}^\lambda: \mathcal E^1(X,\theta) \to \RR$, the $\lambda$-Ding functional:
\begin{flalign}\label{def: twisted_Ding}
\mathcal D_\mu^\lambda(\varphi)=-\frac{1}{\lambda}\log\int_Xe^{-\lambda\varphi} \mathrm{d}\mu-I_\theta(\varphi)\text{ for }\varphi\in\mathcal E^1(X,\theta),
\end{flalign}
where $I_\theta(\cdot)$ is the Monge-Amp\`ere energy (see \eqref{eq: I_def}).  Our starting point is the following result that gives a formula for the slope of the $\lambda$-Ding functional along subgeodesic rays: 

\begin{theorem}\label{mthm: main_radial_ding}  Let $(0,\infty) \ni t \mapsto u_t \in \mathcal E^1(X,\theta)$ be a sublinear subgeodesic ray. Then
\begin{equation}\label{eq: Ding_slope}
\liminf_{t \to \infty} \frac{\mathcal  D_{\mu}^\lambda(u_t)}{t} = -\lim_{t \to \infty} \frac{I_\theta(u_t)}{t} + \sup\{\tau \ \in \RR \ : \  \int_X e^{-\lambda \hat u_\tau} \mathrm{d}\mu <\infty\}.
\end{equation}
\end{theorem}

For the definition of subgeodesic rays $\{u_t\}_t$ and their Legendre transforms $\{\hat u_\tau\}_\tau$ see Definition \ref{def: lin_subgeod} and \eqref{eq: Leg_transf_def_ray_test_curve}. Our $\lambda$-Ding functional is typically not convex along subgeodesics and we overcome this difficulty by building on ideas from \cite[Section 4]{DX20} to prove \eqref{eq: Ding_slope}. We refer to \cite[Theorem 1.3]{Ber16} and \cite[Theorem 5.4]{BBJ21} for similar flavour results.

On the heels of the above theorem it is convenient to introduce  
$$
\mathcal D_{\mu}^\lambda \{u_t\} := \liminf_{t \to \infty} \frac{\mathcal  D_{\mu}^\lambda(u_t)}{t},
$$and we will call this expression the radial $\lambda$-Ding functional of the subgeodesic ray $\{u_t\}_t$. 

Next we introduce a more general delta invariant associated to our data:
\begin{equation}\label{eq: delta_mu_def}
\delta_\mu=\delta_\mu(\{\theta\}):=\inf_E\frac{A_{\chi,\psi}(E)}{S_\theta(E)},
\end{equation}
where $
A_{\chi,\psi}(E):=A_X(E)+\nu(\chi,E)-\nu(\psi,E)$ and the $\inf$ is taken over prime divisors $E$ in smooth bimeromorphic models over $X$. Note that $A_{\chi,\psi}(E)>0$ by \cite[Theorem B.5]{BBJ21}.

The main technical ingredient in the proof of Theorem \ref{thm: main_KE_twist_exist} is the following result, relating $\delta_\mu$ to geodesic semistability of the $\lambda$-Ding functionals:
\begin{theorem}\label{thm:delta=sup-D>0}
With the notation from above we have
$$
\delta_\mu=\sup\{\lambda>0 \ | \ \mathcal{D}_{\mu}^\lambda\{u_t\}\geq 0\text{ for all sublinear subgeodesic ray }u_t\in\mathcal{E}^1(X,\theta)\}.
$$
\end{theorem}
When $\{\theta\}$ is ample and $\chi=0$, this result is proved in \cite{Zh21} using quantization. See also \cite[Proposition 4.5]{XZ20} for a similar statement in the log Fano setting, and \cite[Theorem 3.16]{BJ22} for a related result about ample classes, formulated using non-Archimedean language. Theorem \ref{thm:delta=sup-D>0}  generalizes these results to big cohomology classes in $H^{1,1}(X,\RR)$, yielding a universal analytic interpretation of the delta invariant.

Regarding the much more developed theory of KE metrics on manifolds of general type we note the important work \cite{BEGZ10}, that studies KE equations when $K_X$ is big, in which case a solution always exists. See \cite{BeGu14, EGZ09} for synergies between singular KE metrics and the minimal model program on varieties of general type. Regarding exciting developments on csck metrics we refer to \cite{BJ22,DeLe20} for connections with divisorial notions of K-stability.

Regarding  transcendental K-stability in the K\"ahler setting, we mention the works \cite{DeRo17,Dyr16}  exploring a cohomological notion of stability.

\paragraph{Future directions.} 
\vspace{-0.4cm}
To stay brief, we cannot treat several interesting related questions.

Though the condition $\eta_\psi \geq 0$ allows one to conjecture a Bando--Mabuchi uniqueness theorem \cite{BM87} for twisted KE metrics \cite[page 4]{Bern15}, it does come with a slight limitation. Indeed, when  \eqref{eq: KE_scalar equation} is solvable,  a positive klt current will exist in $c_1(-K_X)$. By Nadel vanishing, this implies that $H^2(X,\mathcal O_X)=0$, hence $\mathbb{ R}$-divisors exhaust the whole space $H^{1,1}(X,\mathbb{R})$. Since neither Theorem \ref{mthm: main_radial_ding} nor \ref{thm:delta=sup-D>0} uses this numerical condition, we wonder if $\eta_\psi\geq0$ can be omitted from Theorem \ref{thm: main_KE_twist_exist} as well. In this direction, \cite[Theorem 2.3]{Zh21} treats the ample case,  suggesting that this is possible.

As pointed out earlier, our definition of the delta invariant is rooted in the non-Archimedean approach to K--stability. As a result, our treatment has an algebraic/non-Archimedean interpretation that  can be expanded when $\{\theta\}$ is the Chern class of a big line bundle. We will explore this in the companion paper \cite{DXZ22}.

%Another natural question is to determine the regularity of solutions to \eqref{eq: KE_scalar equation} in the ample locus of $\{\theta\}$. One would hope that the solution is smooth on this open set. The first step would be to establish this when $\{\theta\}$ is big and nef, akin to \cite[Theorem C]{BEGZ10}. See \cite{KZ22} for  progress on the regularity of conic csc K\"ahler metrics in big and semipositive classes.

As is well known, in the (log) Fano case $\delta> 1$ implies that $X$ can not contain any nontrivial vector fields. See \cite{Berm22} for a related work using Gibbs stability in the big case.  It is desirable to extend our treatment to allow for a non-trivial automorphism group as well. This should be possible after incorporating ideas from the works \cite{His16,Li20YTD}.

%The recent surge of interest in the delta invariant is partially due to its computability for ample line bundles, as evidenced by \cite{AZ20,CZ19, PW18}, just to name a few works. We are very curious if such computations can also be carried out in case of big line bundles. 

Lastly, our approach seems adaptable to the case of canonical KE metrics with prescribed singularity type, as recently studied by Trusiani \cite{Tr20}. If the Ross--Witt Nystr\"om correspondence could be extended to this context, we suspect that the analogue of Theorem \ref{thm: main_KE_twist_exist} could be established, as the convexity of the corresponding Ding functional is known in that setting.

\paragraph{Organization.}\vspace{-0.4cm} In Section \ref{sec:pre} we recall some aspects of finite energy pluripotential theory and complex singularity exponents. In Section \ref{sec:rwncorre} we give a detailed exposition of the Ross--Witt Nystr\"om correspondence between rays and test curves in the big case. In Section \ref{sec:delta=geodesic-stability} we prove Theorem \ref{mthm: main_radial_ding}  and Theorem \ref{thm:delta=sup-D>0}. In Section \ref{sec:big-YTD} we prove Theorem \ref{thm: main_KE_twist_exist}. In Section \ref{sec:log-Fano} we give a short proof of the main result of Li--Tian--Wang \cite{LTW21b} and show how their result in turn implies our Corollary \ref{cor: main_KE_twist_exist}.

\paragraph{Acknowledgments.}\vspace{-0.4cm} 
We thank B. Berndtsson, R. Berman, S. Boucksom,   M. Jonsson, Y. Odaka, G. Tian, M. Xia, C. Xu, X. Zhou, X. Zhu for suggesting simplifications and other improvements to the presentation of the paper. The first named author was partially supported by an Alfred P. Sloan Fellowship and National Science Foundation grant DMS--1846942. The second named author is supported by NSFC grant 12101052 and the
Fundamental Research Funds 2021NTST10 and is grateful to G. Tian for numerous inspiring conversations related to this project.

\section{Preliminaries}
\label{sec:pre}

\subsection{Finite energy pluripotential theory}
\label{sec:psh-theory}
In this short subsection we recall the basics of finite energy pluripotential theory in big cohomology classes. For a more thorough treatment we refer to the articles \cite{BEGZ10}, \cite{DDNL3}, or the recent textbook \cite{GZ17}.

Let $(X,\omega)$ be a compact K\"ahler manifold of dimension $n$ and  $\theta$ a smooth closed $(1,1)$-form. A function $u: X \rightarrow \mathbb{R}\cup \{-\infty\}$ is called quasi-plurisubharmonic (qpsh) if locally $u= \rho + \varphi$, where $\rho$ is smooth and $\varphi$ is a plurisubharmonic (psh) function. We say that $u$ is $\theta$-plurisubharmonic ($\theta$-psh) if it is qpsh and $\theta_u:=\theta+\ddc u \geq 0$ as currents. We let $\PSH(X,\theta)$ denote the space of $\theta$-psh functions on $X$. 

The class $\{\theta\}$ is {\it big} if there exists $\psi\in \textup{PSH}(X,\theta)$ satisfying $\theta_\psi\geq \varepsilon \omega$ for some $\varepsilon>0$. We assume that  $\{\theta\}$ is big throughout this paper, unless specified otherwise.

 Given $u,v \in \textup{PSH}(X,\theta)$,  $u$ is more singular than $v$, (notation: $u \preceq v$) if there exists $C\in \mathbb{R}$  such that $u\leq v+C$. The potential $u$ has the same singularity as $v$ (notation: $u \simeq v$) if $u\preceq v$ and $v\preceq u$. 
The classes $[u]$ of this latter equivalence relation are called \emph{singularity types}. When $\{\theta\}$ is merely big, all elements of $\textup{PSH}(X,\theta)$ are very singular, and we distinguish the potential with minimal singularity:
\begin{equation}
    \label{eq:def-V-theta}
    V_\theta := \sup \{u \in \textup{PSH}(X,\theta) \textup{ such that } u \leq 0\}.
\end{equation}
A function $u\in \PSH(X,\theta)$ is said to have minimal singularity if it has the same singularity type as $V_{\theta}$, i.e., $[u]=[V_\theta]$.

We say that $[u]$ is an \emph{analytic singularity type} if it has a representative $u \in \textup{PSH}(X,\theta)$ that locally can be written as $u = c\log (\sum_{j} |f_j|^2) +g$, where $c>0$, $g$ is a bounded function and the $f_j$ are a finite set of holomorphic functions.
By a fundamental approximation theorem of Demailly there are plenty of $\theta$-psh functions with analytic singularity type. 

The \emph{ample locus} $\mathrm{Amp}(\theta)$ of $\{\theta\}$ is the set of points $x\in X$ such that there exists $u \in \textup{PSH}(X,\theta)$ with analytic singularity type, satisfying $\theta_u\geq\varepsilon\omega$ for some $\varepsilon>0$ and that $u$ is smooth around $x$. In particular, $V_\theta$ is locally bounded on $\mathrm{Amp}(\theta)$.

Let $\theta^1,...,\theta^n$ be smooth closed $(1,1)$-forms  and $\varphi_j \in \textup{PSH}(X,\theta^j)$, $j=1,...n$. Following Bedford-Taylor in the local setting \cite{BT82,BT76}, it has been pointed out in \cite{BEGZ10} that the sequence of positive measures
\begin{equation}\label{eq: k_approx_measure}
\mathbbm{1}_{\bigcap_j\{\varphi_j>V_{\theta^j}-k\}}\theta^{1}_{\max(\varphi_1, V_{\theta^1}-k)}\wedge \ldots\wedge \theta^n_{\max(\varphi_n, V_{\theta^n}-k)}
\end{equation}
converges weakly to the so called \emph{non-pluripolar product} 
$\theta^1_{\varphi_1 } \wedge\ldots\wedge\theta^n_{\varphi_n }.
$
The resulting positive Borel measure does not charge pluripolar sets. In the particular case when $\varphi_1=\varphi_2=\ldots=\varphi_n=\varphi$ and $\theta^1=...=\theta^n=\theta$ we will call $\theta_{\varphi}^n$ the \emph{non-pluripolar Monge-Amp\`ere measure} of $\varphi$, generalizing the usual notion of volume form, when $\theta_{\varphi}$ is a smooth K\"ahler form. 

As a consequence of Bedford-Taylor theory, the measures in \eqref{eq: k_approx_measure} all have total mass less than $\int_X \theta_{V_\theta}^n$, in particular, after letting $k \to \infty$ we notice that $\int_X \theta_{\varphi}^n \leq \int_X \theta_{V_\theta}^n$. What is more, it was proved in \cite[Theorem 1.2]{WN19} that for any $u,v \in \textup{PSH}(X,\theta)$  we have the following monotonicity result for the masses:
if $v \preceq u$ then $\int_X \theta_v^n \leq \int_X \theta_u^n.$

\begin{definition}\label{def: fullmass}We say that $u \in \textup{PSH}(X,\theta)$ is a full mass potential (notation: $u \in \mathcal E(X,\theta)$) if 
$\int_X \theta_u^n = \int_X \theta_{V_\theta}^n =: \vol(\{\theta\}).$
Moreover, we say that $u \in \mathcal E(X,\theta)$ has finite energy (notation: $u \in \mathcal E^1(X,\theta)$) if
$\int_X |u - V_\theta| \theta_{u}^n < \infty.$
\end{definition}

The class $\mathcal E^1(X,\theta)$  plays a central role in the variational theory of complex Monge--Amp\`ere equations, as detailed in \cite{BBEGZ16,BEGZ10} and later works. Here we only mention that the Monge--Amp\`ere energy $I_\theta$ naturally extends to this space with the usual formula:
\begin{equation}\label{eq: I_def}
I_\theta(u) = \frac{1}{\vol(\{\theta\}) (n+1)}\sum_{j = 0}^n \int_X (u - V_\theta) \theta_u^j  \wedge \theta_{V_\theta}^{n-j}, \ \ \ \ u \in \mathcal E^1(X,\theta).
\end{equation}
It is upper semi-continuous (usc) with respect to the $L^1$ topology on $\textup{PSH}(X,\theta)$.

\medskip
We recall some envelope notions that will be useful in this work.
Given any $f : X \to [-\infty,+\infty]$ the starting point is the envelope $P_\theta(f):=\textup{usc}(\sup\{v \in \textup{PSH}(X,\theta), \ v \leq f \})$, where $\mathrm{usc}$ denotes the upper-semicontinuous regularization. Then, for $u,v \in \textup{PSH}(X,\theta)$ we can introduce the ``rooftop envelope'' $P_\theta(u,v):=P_\theta(\min(u,v))$. This allows us to further introduce envelopes with respect to singularity type \cite{RWN14}:
$$P_\theta[u](v) := \textup{usc}\Big(\lim_{C \to +\infty}P_\theta(u+C,v)\Big).$$

It is easy to see that $P_\theta[u](v)$ depends on the singularity type $[u]$. When $v = V_\theta$, we simply write $P[u]:=P_\theta[u]:=P_\theta[u](V_\theta)$ and call this potential the \emph{envelope of the singularity type} $[u]$. It follows from \cite[Theorem 3.8]{DDNL2} that $\theta_{P[u]}^n \leq \mathbbm{1}_{\{P[u] =0\}} \theta^n$. 
Also, by \cite[Proposition 2.3 and Remark 2.5]{DDNL2} we have that $\int_X \theta_{P[u]}^n = \int_X \theta_u^n$.

With the help of these envelopes one can define a complete metric on $\mathcal E^1$. Indeed, as pointed out in \cite[Theorem 2.10]{DDNL1}, for $u,v \in \mathcal E^1(X,\theta)$ we have that $P(u,v) \in \mathcal E^1(X,\theta)$ and the following expression defines a complete metric on $\mathcal E^1(X,\theta)$ \cite[Theorem 1.1]{DDNL3}:
$$d_1(u,v) = I_\theta(u) + I_\theta(v) - I_\theta(P(u,v)).$$
In addition, $d_1$-convergence implies $L^1$-convergence of qpsh potentials \cite[Theorem 3.11]{DDNL3}.

\subsection{Complex singularity exponents and openness}

To start, we point out several notational differences compared to the closely related work of Berman--Boucksom--Jonsson \cite{BBJ21}. Stemming from our choice $\ddc := i \ddbar /2\pi$, our Ding functionals \eqref{def: twisted_Ding} are free of  factors of two compared to \cite[page 4]{BBJ21}. Moreover, for any qpsh function $\varphi$ on $X$ the \emph{complex singularity exponent} attached to our tame measure $\mu:= e^{\chi -\psi} \omega^n$ (recall \eqref{eq:def-mu}) is defined to be
$$
c_\mu[\varphi]:=\sup\bigg\{\lambda>0:\int_Xe^{-\lambda\varphi} \mathrm{d}\mu <\infty\bigg\}.
$$
And our definition of the Lelong number of a qpsh function $v$ at $x \in X$ is 
\begin{equation}
    \label{eq:def-Lelong-number}
    \nu(v,x) := \sup \{ c > 0  \textup{ s.t. }  v(z) - c \log|z-x|^2 \textup{ is bounded above near } x\}. 
\end{equation}
Given an irreducible analytic subset $E \subset X$, $\nu(v,E)$ is equal to the Lelong number of $v$ at a very general point of $E$ (by \cite{Siu74}).

On the level of potentials, our definition of singularity exponent and Lelong number differs from the ones used in \cite{BBJ21} by a factor of two. However on the level of $(1,1)$-currents, all definitions agree due to our choice of $\ddc$. As an upside, using our terminology, all the formulas/inequalities we derive in this work contain no extra factors of two.

Openness theorems go back to the work of Berndtsson \cite[Theorem 4.4]{Bern15o} on  Demailly--Koll\'ar's openness conjecture \cite{DK01}.
Not long after, Guan--Zhou \cite{GuZh15}  solved the strong openness conjecture of Demailly \cite{Dem01}, which will be used multiple times in this work. The version  below follows from (the proof of) \cite[Corollary B.2]{BBJ21} and the effective version of the strong openness theorem \cite[\S 3.3]{GuZh15} (see also \cite[Main Theorem (2)]{Hi14} and \cite[Corollary 1.2]{GLZ16}):

\begin{theorem}\label{thm: openness} Suppose that $\lambda >0$ and $u,u_j$ are qpsh functions on $X$ such that $u_j \nearrow u$ a.e., and $\int_X e^{-\lambda u} \mathrm{d}\mu <\infty$. Then there exists  $j_0$ such that  $\int_X e^{-\lambda u_j} \mathrm{d}\mu <\infty$ for $j \geq j_0.$
\end{theorem}

The above theorem implies that $c_\mu[\varphi]$ is always positive. The following result (that follows from \cite{BBJ21,BFJ08} and Guan--Zhou openness \cite{GuZh15}) gives a more precise valuative characterization of $c_\mu[\cdot]$.

\begin{theorem}
\label{thm:c=inf-A/S}
For any qpsh function $\varphi$ on $X$ one has
\begin{equation}\label{eq: c_mu_eq}
c_\mu[\varphi] = \inf _E \frac{A_{\chi,\psi}(E)}{\nu(\varphi,E)},
\end{equation}
where $
A_{\chi,\psi}(E):=A_X(E)+\nu(\chi,E)-\nu(\psi,E)$ and the infimum is taken over all prime divisors $E$ in smooth bimeromorphic models over $X$.
\end{theorem}

If on the right hand side of \eqref{eq: c_mu_eq} we have $\nu(\varphi,E)=0$, then we use the convention $1/0 = \infty$. No ambiguity will arise from this, as follows from the proof below.

\begin{proof} 
Let  $\lambda\in(0,c_\mu[\varphi])$. By \cite[Theorem B.5]{BBJ21}, for small enough $\varepsilon>0$ we have
$$
\nu(\chi,E)+A_X(E)\geq(1+\varepsilon)(\nu(\psi,E)+\lambda\nu(\varphi,E))\geq\nu(\psi,E)+(1+\varepsilon)\lambda\nu(\varphi,E)
$$
for any prime divisor $E$ over $X$. Letting $\varepsilon \searrow 0$ and then $\lambda\nearrow c_\mu[\varphi]$, we arrive at 
$$
c_\mu[\varphi] \leq \inf _E \frac{A_{\chi,\psi}(E)}{\nu(\varphi,E)}.
$$

For the other direction, let $\lambda\in\big(0,\inf_E\frac{A_{\chi,\psi}(E)}{\nu(\varphi,E)}\big)$ (note that $\inf_E\frac{A_{\chi,\psi}(E)}{\nu(\varphi,E)}$ is indeed positive by the previous step). 
So there is some $\varepsilon>0$ such that
$$
\nu(\chi,E)+A_X(E)\geq\nu(\psi,E)+(1+\varepsilon)\lambda\nu(\varphi,E)
$$
for all $E$.
On the other hand, we have $\int_Xe^{\chi-\psi}\omega^n=\int_X\mathrm{d}\mu<\infty$. Using \cite[Theorem B.5]{BBJ21} again, for any small $\tau>0$ we have that
$$
\nu(\chi,E)+A_X(E)\geq(1+\tau)\nu(\psi,E)
$$
holds for all $E$. Summarizing, we find that
$$
(1+\tau)(\nu(\chi,E)+A_X(E))\geq(1+\tau+\tau^2)\nu(\psi,E)+(1+\varepsilon)\lambda\nu(\varphi,E).
$$
Or equivalently,
$$
\nu(\chi,E)+A_X(E)\geq\frac{1+\tau+\tau^2}{1+\tau}\nu(\psi,E)+\frac{1+\varepsilon}{1+\tau} \lambda \nu(\varphi,E).
$$
Finally, choosing $\tau$ such that $\tau<\varepsilon$, we get
$
\nu(\chi,E)+A_X(E)\geq(1+\varepsilon^\prime)(\nu(\psi,E)+\lambda\nu(\varphi,E))
$
for some $\varepsilon^\prime>0$, which holds for all $E$. Using \cite[Theorem B.5]{BBJ21} one more time we obtain $c_\mu[\varphi]\geq\lambda$, implying the other direction.
\end{proof}

Theorem \ref{thm:c=inf-A/S} has the following consequence.
\begin{lemma}
\label{lem:1/c-convex}
Let $\{\psi_{\tau}\}_{\tau\in I}$ be a family of qpsh functions on $X$, where $I\subset\RR$ is some connected open interval. Assume that $I\ni\tau\mapsto\psi_\tau(x)$ is concave for any $x\in X$. Then the map
$
I\ni\tau\mapsto 1/c_\mu[\psi_\tau]
$
is convex, so in particular it is also continuous.
\end{lemma}

\begin{proof}
For $\tau_1,\tau_2\in I$ and $\lambda\in[0,1]$, concavity reads
$
\psi_{\lambda\tau_1+(1-\lambda)\tau_2}\geq\lambda\psi_{\tau_1}+(1-\lambda)\psi_{\tau_2}.
$
Let $E$ be any prime divisor over $X$. Then
$
\nu(\psi_{\lambda\tau_1+(1-\lambda)\tau_2},E)\leq\lambda\nu(\psi_{\tau_1},E)+(1-\lambda)\nu(\psi_{\tau_2},E),
$
so that
$$
\frac{\nu(\psi_{\lambda\tau_1+(1-\lambda)\tau_2},E)}{A_{\chi,\psi}(E)}\leq\frac{\lambda\nu(\psi_{\tau_1},E)}{A_{\chi,\psi}(E)}+\frac{(1-\lambda)\nu(\psi_{\tau_2},E)}{A_{\chi,\psi}(E)}.
$$
Taking sup over all $E$, by Theorem \ref{thm:c=inf-A/S}, we conclude the result.
\end{proof}

Lastly we recall a consequence of Theorem \ref{thm: openness} and \cite[Theorem 1.3]{DDNL2}:

\begin{proposition}
\label{prop:c-u=c-P-u}
$c_\mu[u] = c_\mu[P[u]]$ for any $u \in \textup{PSH}(X,\theta)$. In particular $c_\mu[v] = c_\mu[V_\theta]$ for any $v \in \mathcal E(X,\theta)$.
\end{proposition}

\begin{proof} Recall that $P(V_\theta, u + C)\nearrow P[u]$ a.e. As $[P(V_\theta, u + C)] = [u]$  for any $C\in \RR$, Theorem \ref{thm: openness} immediately gives that $c_\mu[u] = c_\mu[P[u]]$. By \cite[Theorem 1.3]{DDNL2} we have that $P[v] = V_\theta$ for any $v \in \mathcal E(X,\theta)$. The last statement immediately follows.
\end{proof}

\section{The extended Ross--Witt Nyström correspondence}\label{sec:rwncorre}

We give a  precise correspondence between finite energy geodesic rays and certain maximal test curves in the big case. This theory was initiated by Ross and Witt Nystr\"om in \cite{RWN14} in the K\"ahler case, and developed further in \cite{DDNL3} and \cite{DX20}. 

\begin{definition}\label{def: lin_subgeod}
A sublinear subgeodesic ray is a subgeodesic $(0,\infty) \ni t \mapsto u_t \in \PSH(X, \theta)$ (notation: $\{u_t\}_t$) such that $u_t \to_{L^1} u_0: = V_\theta$ as $t \to 0$, and there exists $C\in \mathbb{R}$ such that $u_t(x) \leq C t, \ t\geq 0$, $x \in X$. In addition, $\{u_t\}_t$ is of finite energy if $u_t \in \mathcal E^1(X,\theta), \ t \geq 0$.
\end{definition}
Using $t$-convexity, we obtain some immediate properties of sublinear subgeodesic rays:
\begin{lemma}\label{lem: -inf_est_subgeod} Suppose that $\{u_t\}_t$ is a sublinear subgeodesic ray. Then the set $\{u_t >-\infty \}$ is the same for any $t >0$. In particular, for any $x \in X$ the curve $t \mapsto u_t(x)$ is either finite and convex on $(0,\infty)$, or equal to $-\infty$ on this interval.
\end{lemma}

A \emph{psh geodesic ray} is a sublinear subgeodesic ray that additionally satisfies the following maximality property: for any $0 < a < b$, the subgeodesic $(0,1) \ni t \mapsto v^{a,b}_t:=u_{a(1-t) + bt} \in \PSH(X,\theta)$ can be recovered in the following manner:
\begin{equation}\label{eq: vabt_eq}
v^{a,b}_t:=\sup_{h \in \mathcal{S}}{h_t}\,,\quad t \in [0,1]\,,
\end{equation}
where $\mathcal{S}$ is the set of subgeodesics $(0, 1) \ni  t \to h_t \in \PSH(X,\theta)$ with $\lim_{t \searrow 0}h_t\leq  u_a$ and $\lim_{t \nearrow 1}h_t\leq u_b$. 
The space of \emph{finite energy (psh) geodesic rays} will be denoted by 
$$\mathcal R^1(X,\theta).$$

\begin{remark} \label{rem: sup_ray}Let $\{u_t\}_t$ be a psh geodesic ray. Due to \cite[(13)]{DDNL3}  the map $t \mapsto \sup_X u_t=\sup_X (u_t - V_\theta)=\sup_{\textup{Amp}\{\theta\}} (u_t - V_\theta), \ t > 0$ is linear, using an approximation argument via decreasing geodesic segments with minimal singularity type.
\end{remark}

Making small tweaks to \cite[Definition 5.1]{RWN14}, we now give the definition of test curves:

\begin{definition}
A map $\mathbb{R}\ni \tau \mapsto \psi_\tau\in \PSH(X, \theta)$ is a \emph{psh test curve}, denoted $\{\psi_\tau\}_{\tau}$, if \vspace{0.15cm}\\
(i) $\tau\mapsto \psi_\tau(x)$ is concave, decreasing and usc for any $x\in X$. \vspace{0.15cm}\\
(ii) $\psi_\tau\equiv -\infty$ for all $\tau$ big enough, and $\psi_\tau$ increases a.e. to $V_\theta$ as $\tau \to -\infty$.
\end{definition}
Note that this definition is more general than the one in \cite{RWN14} (where the authors only considered potentials with small unbounded locus), more general than the one in \cite{DDNL3} (where the authors considered only bounded test  curves, see below), and more general than the one in \cite[Section 3.1]{DX20} (where the authors only consider K\"ahler classes). Moreover, condition (ii) allows for the introduction of the following constant:
\begin{equation}\label{eq: psi+_tau_def}
\tau_\psi^+ := \inf\{\tau \in \mathbb{R} : \psi_\tau \equiv -\infty\}\,.
\end{equation}

We adopt the following convention: psh test curves will always be parametrized by $\tau$, whereas rays will be parametrized by $t$. Hence $\{\psi_t\}_t$ will always refer to some kind of ray, whereas  $\{\phi_\tau\}_\tau$ will refer to some type of test curve. As pointed out below, rays and test curves are dual to each other, so one should think of the  parameters $t$ and $\tau$ to be dual to each other as well.

\begin{definition}\label{def: test_curves}
\leavevmode A psh test curve $\{\psi_\tau \}_\tau$  can have the following properties: \vspace{0.1cm}\\
(i) $\{\psi_{\tau}\}_\tau$ is \emph{maximal} if $P[\psi_\tau] =\psi_\tau$ for any $\tau \in \mathbb{R}$.\vspace{0.1cm}\\
(ii) $\{\psi_{\tau}\}_\tau$ is a \emph{finite energy test curve} if
    \begin{equation*}\label{eq: fetestcurve_def}
        \int_{-\infty}^{\tau^+_\psi} \left( \int_X \theta_{\psi_\tau}^n-\int_X \theta_{V_\theta}^n \right) \,\mathrm{d}\tau >-\infty\,.
    \end{equation*}
(iii) We say $\{\psi_{\tau}\}_\tau$ is \emph{bounded} if $\psi_\tau = V_\theta$ for all $\tau$ negative enough. In this case, one can introduce the following constant, complementing \eqref{eq: psi+_tau_def}:
 \begin{equation*}\label{eq: psi-_tau_def}
\tau_\psi^- := \sup\left\{\,\tau \in \mathbb{R} : \psi_\tau \equiv V_\theta\,\right\}\,.
\end{equation*}
\end{definition}
In the above definition, we followed the convention $P[-\infty]=-\infty$.
Note that bounded test curves are clearly of finite energy.

 We recall the \emph{Legendre transform}, that will establish the duality between various types of maximal test curves and geodesic rays. Given a convex function $f:[0, +\infty)\rightarrow \mathbb{R}$, its Legendre transform is defined as 
 \begin{equation*}
 \hat f(\tau):= \inf_{t \geq 0} (f(t)-t\tau)=\inf_{t > 0} (f(t)-t\tau)\,,\quad \tau\in \mathbb{R}\,.
 \end{equation*}
The \emph{(inverse) Legendre transform} of a decreasing concave function $g:\mathbb{R}\rightarrow \mathbb{R}\cup \{-\infty\}$ is
\begin{equation*}\label{eq: inverse_Lag_tran_def}
\check{g}(t):=\sup_{\tau \in \mathbb{R}} (g(\tau)+t\tau)\,, \quad t \geq 0\,.
\end{equation*}
There is a sign difference in our choice of Legendre transform compared to the convex analysis literature, however this choice will be more suitable for our discussion.

We recall that, for every $\tau \in \mathbb{R}$ we have  $\hat{\check{g}}(\tau) \geq g(\tau)$ with equality if and only if $g$ is additionally $\tau$-usc. Similarly, $\check{\hat{f}}(t) \leq f(t)$ for all $t\geq 0$ with equality if and only if $f$ is $t$-lsc.   In general, $\hat{\check{g}}$ is the $\tau$-usc envelope of $g$, and $\check{\hat{f}}$ is the $t$-lsc envelope of $f$. These properties are called the involution property of the Legendre transform. 

Starting with a psh test curve $\{\psi_\tau\}_\tau$, our goal is to construct a geodesic/subgeodesic ray by taking the $\tau$-inverse Legendre transform. The first step is the next proposition, essentially proved in \cite{Da17}:

\begin{proposition} \label{lem: Legendre_usc} Suppose $\{\psi_\tau\}_\tau$ is a psh test curve. Then $\sup_{\tau}(\psi_\tau(x) + t\tau)$ is usc with respect to $(t,x)\in(0,\infty)\times X$.
\end{proposition}

Since $\tau^+_\psi <\infty$ and $\psi_\tau \leq V_\theta, \tau \in \mathbb R$, we note that $ \sup_{\tau} (\psi_\tau + t\tau) \leq V_\theta + t\tau^+_\psi$ for $t \geq 0$. Even if true, it is not clear how to interpret upper semi-continuity at $(t,x) \in \{0\} \times X$. This is due to the fact that $V_\theta(x) =\sup_{\tau}\psi_\tau(x) $  a.e. $x \in X$, but not everywhere (!) in our definition of psh test curve.

\begin{proof} Let $S = \{ \textup{Re } s > 0\}$ be the right open half plane. 
We consider the usual complexification of the inverse Legendre transform: 
\[
u(s,z) := \sup_{\tau} (\psi_\tau(z) +  \tau\textup{Re } s)\,,\quad (s,z) \in S \times X\,.
\]
Also, $u_t(x): = u(t,x) \leq V_\tau +  t\tau^+_\psi \leq t\tau^+_\psi, t >0$. Clearly, $\usc u \in \PSH(S\times X,\pi^*\omega)$, where $\usc u$ is the usc regularization of $u$ on $S \times X$, where $\pi:S\times X\rightarrow X$ is the natural projection.
It will be enough to prove that  $\usc u=u$. 

We introduce the set $E = \{ u < \usc u\} \subseteq S \times X$. As both $u$ and $\usc u$ are $\mathbb{R}$-invariant in the imaginary direction of $S$, it follows that $E$ is also $\mathbb{R}$-invariant, i.e., there exists $B \subseteq (0,\infty) \times X$ such that $E = B \times \mathbb R$.

As $E$ has Monge--Ampère capacity zero, it follows that $E$ has Lebesgue measure zero. By Fubini's theorem $B \subseteq (0,\infty) \times X$ has Lebesgue measure zero as well. For $z \in X$, we introduce the $z$-slices of $B$:
\[
B_z = B \cap \left((0,\infty) \times \{z\}\right)\,.
\]
By Fubini's theorem again, we have that $B_z$ has Lebesgue measure $0$ for all $z \in X \setminus F$, where $F \subseteq X$ is some set of Lebesgue measure $0$. 

By slightly increasing $F$, but not its zero Lebesgue measure, we can additionally assume that $u_t(z)> -\infty$ for all $t >0$ and $z \in X \setminus F$ (indeed, at least one potential $\psi_\tau$ is not identically equal to $-\infty$). 

Let $z \in X \setminus F$. We argue that $B_z$ is in fact empty. By our assumptions on $F$, both maps $t \mapsto u_t(z)$ and $t \mapsto (\usc u)(t,z)$ are locally bounded and convex (hence continuous) on $(0,\infty)$. As they agree on the dense set $(0,\infty) \setminus B_z$, it follows that they have to be the same, hence $B_z=\emptyset$. This allows to conclude that 
\begin{equation}\label{eq: a.e._id1}
\inf_{ t> 0} [u_t(x) - \tau t] = \chi_\tau :=\inf_{ t> 0} \left[{(\usc u)}(t,x) - \tau t\right]\,, \quad \tau \in \mathbb{R}  \textup{ and } z \in X \setminus F\,.
\end{equation}
By duality of the Legendre transform $\psi_\tau(x) = \inf_{t > 0} [u_t(x) - t\tau]$ for all  $x \in X$ and $\tau \in \mathbb R$ (using the $\tau$-usc property of $\tau \mapsto \psi_\tau$). From this and \eqref{eq: a.e._id1} it follows that $\psi_\tau=\chi_\tau$ a.e. on $X$, for all $\tau \in \mathbb{R}$. Since both $\psi_\tau$ and $\chi_\tau$ are $\theta$-psh (the former by definition, the latter by Kiselman's minimum principle \cite[Theorem~I.7.5]{De12}), it follows that $\psi_\tau\equiv\chi_\tau$ for all $\tau \in \mathbb{R}$. 

Consequently, applying the $\tau$-Legendre transform to the $\tau$-usc and $\tau$-concave curves $\tau \mapsto \psi_\tau$ and $\tau \mapsto \chi_\tau$, we obtain that $t \mapsto u_t(x)$ is the $t$-lsc envelope of $t \mapsto \usc u(t,x)$ on $(0,\infty)$ for all $x \in X$. hence, Lemma   \ref{lem: -inf_est_subgeod}  gives that $u_t(x) = \textup{usc } u(t,x)$ on $(0,\infty) \times X$.
\end{proof}

Given a sublinear subgeodesic ray $\{\phi_t\}_t$ (psh test curve $\{\psi_\tau\}_\tau$), we can associate its (inverse) Legendre transform at $x \in X$ as 
\begin{equation}\label{eq: Leg_transf_def_ray_test_curve}
    \begin{aligned}
\hat \phi_\tau(x) := \inf_{t>0}(\phi_t(x) - t\tau)\,,& \quad \tau \in \mathbb R\,,\\
\check \psi_t(x) := \sup_{\tau \in \mathbb R}(\psi_\tau(x) + t\tau)\,,& \quad t> 0\,.
\end{aligned}
\end{equation}

Our next theorem describes a duality between various types of rays and maximal test curves, extending various particular cases from \cite{DDNL3,RWN14}:
\begin{theorem}\label{thm: max_test_curve_ray_duality}
The Legendre transform $\{\psi_\tau\}_\tau \mapsto \{\check \psi_t\}_t$ gives a bijective map with inverse $\{\phi_t\}_t \mapsto \{\hat \phi_\tau\}_\tau$ between:\vspace{0.1cm}\\
(i) psh test curves and sublinear subgeodesic rays,\vspace{0.1cm}\\
(ii) maximal psh test curves and psh geodesic rays,\vspace{0.1cm}\\
(iii)\cite{DDNL3,RWN14} maximal bounded test curves and geodesic rays with minimal singularity type. In this case, we additionally have that
        $
        V_\theta + \tau_\psi^- t \leq {\check \psi_t}\leq V_\theta + \tau_\psi^+t\,, \quad t \geq 0\,.
        $
        
\end{theorem}
\begin{proof}

We prove (i). This is essentially \cite[Proposition~4.4]{DDNL3}, where an important particular case was addressed. Let $\{\psi_\tau\}_\tau$ be a psh test curve. Then $\check \psi_t \in \PSH(X, \theta)$ for all $t>0$ due to  Proposition~\ref{lem: Legendre_usc}. We also see that $ \check \psi_t \leq V_\theta + t \tau^+_\psi$, and $\check \psi_t \to_{L^1} V_\theta$ as $t \to 0$, proving that $\{\check{\psi}_t\}_t$ is a sublinear subgeodesic.

For the reverse direction, let $\{\phi_t\}_t$ be a sublinear subgeodesic ray. Then $\hat \phi_\tau \in \PSH(X, \theta)$ or $\hat \phi_\tau \equiv -\infty$ for any $\tau \in \mathbb R$ due to Kiselman's minimum principle. By properties of Legendre transforms and Lemma~\ref{lem: -inf_est_subgeod}, we get that $\tau \mapsto \hat \phi_\tau(x)$ is $\tau$-usc, $\tau$-concave and decreasing. Due to sublinearity of $\{\phi_t\}_t$ we get that $\hat \phi_\tau \equiv -\infty$ for $\tau$ big enough. Lastly $\psi_\tau \nearrow V_\theta$ a.e. as $\tau \to -\infty$, since $\phi_t \to_{L^1} V_\theta$ as $t \to 0$.

We prove (ii). Let $\tau \in \RR$ and $\{u_t\}_t$ a psh geodesic ray. From \cite[Propisition 5.1]{Da17} (that only uses the maximum principle \eqref{eq: vabt_eq}) we obtain that $\hat u_\tau = P[\hat u_\tau](V_\theta)= P[\hat u_\tau](0)=P[\hat u_\tau]$. Since $\{u_t\}_t$ is sublinear,  the curve $\{ \hat u_\tau\}_\tau$ is a maximal psh test curve. 

Conversely, let $\{\psi_\tau\}_\tau$ be a maximal psh test curve. We will show that the sublinear subgeodesic $\{\check \psi_t\}_t$ is a psh geodesic ray. By elementary translation properties of the Legendre transform we can assume that $\tau^+_\psi =0$, in particular $\{\check \psi_t\}_t$ is $t$-decreasing. 

Now assume by contradiction that $\{\check \psi_t\}_t$ is not a psh geodesic ray. Comparing with \eqref{eq: vabt_eq}, there exists $0 < a < b$ such that 
\[
\check \psi_{(1-t)a + tb} \lneq \chi_t:=\sup_{h \in \mathcal{S}} h_t\,,\quad t \in [0,1]\,,
\]
where $\mathcal{S}$ is the set of subgeodesics $(a, b) \ni  t \mapsto h_t \in \PSH(X, \theta)$  satisfying $\displaystyle\lim_{t \to a+} h_t \leq  \check \psi_a$ and $\displaystyle\lim_{t \to b-} h_t \leq  \check \psi_b$. Now let $\{\phi_t\}_t$ be the sublinear subgeodesic such that $\phi_t := \check \psi_t$ for $t \not\in (a,b)$ and $\phi_{a(1-t) + bt} := \chi_t$ otherwise.

Trivially, $\check \psi_t \leq \phi_t \leq 0$, hence by duality, $\psi_\tau \leq \hat \phi_\tau \leq 0$ and $\tau^+_\psi = \tau^+_{\hat \phi}=0$.  However, comparing with \eqref{eq: Leg_transf_def_ray_test_curve}, we claim that $ \hat \phi_\tau \leq  \psi_\tau + \tau(a-b)$ for any $\tau \in \mathbb R$. Since $\tau^+_\psi = \tau^+_{\hat \phi}=0$, we only need to show this for $\tau \leq 0$. For such $\tau$ we indeed have 
\[
\inf_{t \in [a,b]}(\phi_t - t \tau) \leq \phi_b - b \tau = \check \psi_b - b \tau  \leq \inf_{t \in [a,b]}(\check \psi_t - t \tau) +  (a-b) \tau\,,
\]
where in the last inequality we used that $t \mapsto \check \psi_t$ is decreasing.

By the maximality of $\{\psi_\tau\}_\tau$, we obtain that $\hat \phi_\tau \leq P[\psi_\tau] = \psi_\tau$. As we already pointed out the reverse inequality above, we conclude that $\hat \phi_\tau = \psi_\tau$. The (inverse) Legendre transform now gives that $\check \psi_t=\phi_t$, a contradiction. Hence $\{\psi_t\}_t$ is a psh geodesic ray.

The duality of  (iii)  is simply \cite[Theorem~1.3]{DDNL3}.
\end{proof}

Given a finite energy (sub)geodesic ray $\{u_t\}_t$ we know that $t \to I_\theta(u_t)$ is (convex) linear \cite[Theorem 3.12]{DDNL1}, allowing to introduce the following radial Monge--Amp\`ere energy:
\begin{equation}\label{eq: I-rad-subgeodesic}
I_\theta\{u_t\} := \lim_{t \to \infty} \frac{I_\theta(u_t)}{t}.
\end{equation}

Before proving the duality between maximal finite energy test curves and rays, we point out the following approximation result:

\begin{proposition} \label{prop: psh_ray_bounded_approx}Let $\{ u_t\}_t$ be a psh geodesic ray. Then there exists a sequence of geodesic rays $\{u^j_t\}$ with minimal singularity type such that $u^j_t \searrow u_t$ and $\theta_{\hat u^j_\tau}^n \to \theta_{\hat u_\tau}^n$ weakly for $j \to \infty$ for all $t \geq 0$ and $\tau \in \RR$.
\end{proposition}

\begin{proof} 
In the K\"ahler case, this result is a particular case of \cite[Theorem 4.5]{DL20}. As the argument in the big case is virtually the same, we will be brief, and only point out how the sequence of approximating rays $\{u^j_t\}_t$ is constructed.

We can assume without loss of generality that $\tau^+_{\hat u} = 0$. For $j \in \NN, \tau < 0$,  set 
	$$
	\psi^{j}_{\tau}(x) := \Big(1-\max\Big(0, 1+\frac{\tau}{j}\Big)\Big) V_\theta + \max\Big(0, 1+\frac{ \tau}{j}\Big) \hat u_{\tau}, \ \text{and} \ \hat u^{j}_{\tau} := P[\psi^{j}_{\tau}]. 
	$$
	We define $\hat u^{j}_0:= \lim_{\tau \to 0^-} \hat u^{j}_{\tau}$. 
	
Since $\tau \to \hat u_{\tau}$ is $\tau$-concave, $\tau$-decreasing, and $\hat u_{\tau}\leq V_\theta$, it is elementary to see that $\tau \to \psi^{j}_{\tau}$ is also $\tau$-concave and $\tau$-decreasing. By elementary properties of $P[\cdot]$ we get that $\tau \to \hat u^{j}_{\tau}$ is also $\tau$-concave and $\tau$-decreasing (see the proof of \cite[Proposition 4.6]{DDNL3}).

Arguing the same way as in the proof of \cite[Theorem 4.5]{DL20} we conclude that $u^j_t \searrow u_t$, $\hat u^j_\tau \searrow \hat u_\tau$, and $\theta_{\hat u^j_\tau}^n \to \theta_{\hat u_\tau}^n$ weakly for $j \to \infty$ for all $t \geq 0,\ \tau \in \RR$.
\end{proof}

Finally we prove the Ross--Witt Nystr\"om correspondence between maximal finite energy test curves and the finite energy rays of $\mathcal R^1(X,\theta)$.

\begin{theorem} \label{thm: max_fin_en_test_curve_ray_duality}
        The Legendre transform $\{\psi_\tau\}_\tau \mapsto \{\check \psi_t\}_t$ gives a bijective map with inverse $\{\phi_t\}_t \mapsto \{\hat \phi_\tau\}_\tau$ between
         maximal finite energy test curves and finite energy geodesic rays. In this case, we additionally have that
        \begin{equation}\label{eq: I_RWN_form}
            I_\theta \{\check \psi_t\}=\frac{1}{\vol(\{\theta\})}\int_{-\infty}^{\tau^+_\psi} \left(\int_X \theta_{\psi_\tau}^n-\int_X \theta_{V_\theta}^n \right) \,\mathrm{d}\tau+\tau_{\psi}^+\,.
        \end{equation}
\end{theorem}

\begin{proof}
As previously, we may assume that $\tau_{\psi}^+=0$. As a preliminary result, in Proposition~\ref{prop: RWN_I_formula} below we prove \eqref{eq: I_RWN_form} for bounded maximal test curves.

Given a finite energy maximal test curve $\{\psi_\tau\}_\tau$, we know that $\{\check \psi_t\}_t$ is a psh geodesic ray. 
By Proposition \ref{prop: psh_ray_bounded_approx} above there exists geodesic rays $\{\check \psi^k_t\}_t$ with minimal singularity type such that $\check \psi^k_t \searrow \check \psi_t$ for any $t 
\geq 0$, and $\int_X \theta^n_{\psi^k_\tau } \searrow \int_X \theta^n_{\psi_\tau }$ for any $\tau < \tau^+_\psi = \tau^+_{\psi^k}=0$.    

By basic properties of the Monge--Amp\`ere energy we have that $I_\theta\{\psi^k_t\} = I_\theta(\psi^k_1) \to I_\theta(\psi_1)=I_\theta\{\psi_t\}.$ By Proposition~\ref{prop: RWN_I_formula} below
\[
I_\theta\{\check \psi^k_t\}=\frac{1}{\vol(\{\theta\})}\int_{-\infty}^{0} \left( \int_X \theta_{\psi^k_\tau}^n-\int_X \theta_{V_\theta}^n\right)\,\mathrm{d}\tau\,.
\]

The right hand side is bounded from below, since $\{\psi_\tau\}_\tau$ is a finite energy test curve.
Since $\int_X \theta^n_{\psi^k_\tau } \searrow \int_X \theta^n_{\psi_\tau }$, we can take the limit on both the left and right hand side, to arrive at \eqref{eq: I_RWN_form}, also implying that $\{\check \psi_t\}_t$ is a finite energy geodesic ray. 

Conversely, assume that $\{\phi_t\}_t$ is a finite energy geodesic ray, with decreasing approximating sequence of rays $\{\phi_t^k\}_t$, as detailed above. For similar reasons we have
$
I_\theta\{\phi^k_t\}=\frac{1}{\vol(\{\theta\})}\int_{-\infty}^{0} \left( \int_X \theta_{\hat \phi^k_\tau}^n-\int_X \theta_{V_\theta}^n\right)\,\mathrm{d}\tau\,.
$
Since ${I}\{\phi_t^k\} \searrow I_\theta\{\phi_t\}$,  the monotone convergence theorem gives that \eqref{eq: I_RWN_form} holds for $\{\hat \phi_\tau\}_{\tau}$, finishing the proof.
\end{proof}

As promised, to complete the argument of Theorem \ref{thm: max_test_curve_ray_duality}, we prove the next proposition, whose argument can be extracted from \cite[Section~6]{RWN14} with additional references to \cite{DDNL2}. We recall the precise details here as the results of \cite{RWN14} were proved in the context of potentials with small unbounded locus.

\begin{proposition}\label{prop: RWN_I_formula} Suppose that $\{\psi_\tau\}_\tau$ is a bounded maximal test curve with $\tau^+_\psi = 0$. Then
\begin{equation}\label{eq: I_bounded_formula_RWN}
\frac{I_\theta(\check \psi_t)}{t}=I_\theta\{\check \psi_t\}= \frac{1}{\vol(\{\theta\})}\int_{-\infty}^{0} \left( \int_X \theta_{\psi_\tau}^n-\int_X \theta^n_{V_\theta}\right)\,\mathrm{d}\tau\,, \quad t >0\,. 
\end{equation} 
\end{proposition}
\begin{proof} Without loss of generality we assume that $\vol(\{\theta\})=1$. For $N\in \mathbb{Z}_+,M \in \mathbb{Z}$ and $t >0$, we introduce the following:
\begin{equation}\label{eq: Psi_MN_def}
\check{\psi}^{N,M}_t := \max_{\substack{k \in \mathbb{Z} \\ k \leq M}} \left(\psi_{k/2^N} +  tk/2^N\right)\,.
\end{equation}
It is clear that $\check \psi_t^{N,M}  \in \PSH(X, \theta)$ has minimal singularity type, since it is a maximum of a finite collection of $\theta$-psh potentials (indeed, $\{\psi_\tau\}_\tau$ is a bounded test curve). Following \cite[Lemma 6.6]{RWN14}, we now argue that
\begin{equation}\label{eq: diff_eq_I}
 \frac{t}{2^N} \int_X \theta^n_{\psi_{(M+1)/2^N}}\leq I_\theta(\check \psi_t^{N,M+1})-I_\theta(\check \psi_t^{N,M})\leq \frac{t}{2^N} \int_X \theta^n_{\psi_{M/2^N}}\,.
\end{equation}
Indeed, by \cite[Theorem 2.4(iii)]{DDNL3},
\begin{equation}\label{eq: first_I_ineq}
\int_X \left(\check{\psi}_t^{N,M+1}-\check{\psi}_t^{N,M}\right)\,\theta^n_{\check \psi_t^{N,M+1}}\leq I_\theta(\check{\psi}_t^{N,M+1})-I_\theta(\check{\psi}_t^{N,M})\leq \int_X \left(\check{\psi}_t^{N,M+1}-\check{\psi}_t^{N,M}\right)\,\theta^n_{\check \psi_t^{N,M}}\,.
\end{equation}
Clearly $\check{\psi}_t^{N,M+1} \geq \check{\psi}_t^{N,M}$. Using the $\tau$-concavity of $\tau \mapsto \psi_\tau$ we claim that 
\begin{flalign}\label{eq: set_equal}
 U_t := \left\{\,\check \psi^{N,M+1}_t-\check \psi^{N,M}_t>0\,\right\} \cap \textup{Amp}(\theta)= \left\{\psi_{(M+1)/2^N} +  \frac{t}{2^N} -\psi_{M/2^N}>0\right\} \cap \textup{Amp}(\theta).
\end{flalign}
Indeed, if $\psi^{N,M+1}_t(x)>\check \psi^{N,M}_t(x)$ for some $x \in \textup{Amp}(\theta)$, then the maximum defining $\psi^{N,M+1}_t$ (see \eqref{eq: Psi_MN_def}) is 'uniquely' realized for $k = M+1$, or equivalently,
\begin{equation}\label{eq: many_est}
\psi_{(M+1)/2^N}(x) +  t \frac{M+1}{2^N} > \psi_{k/2^N}(x)+  t \frac{k}{2^N}, \textup{ for all } k \leq M
\end{equation}
The inequality for $k=M$ is equivalent to 
\begin{equation}\label{eq: one_est}
\psi_{(M+1)/2^N}(x) +  \frac{t}{2^N} -\psi_{M/2^N}(x)>0,
\end{equation}
hence the set on the left hand side of \eqref{eq: set_equal} is contained on the right hand side. To argue conversely, by the above, we only have to prove that \eqref{eq: one_est} implies \eqref{eq: many_est}. This is where concavity of $\tau \mapsto u_\tau(x) + t \tau$ is used for some fixed $x \in \textup{Amp}(\theta)$ (c.f. \cite[Lemma 6.5]{RWN14}). More precisely,  \eqref{eq: one_est} implies that
\begin{equation}\label{eq: many_est_2}
\psi_{M/2^N}(x) +  t \frac{M}{2^N} > \psi_{k/2^N}(x)+  t \frac{k}{2^N}, \textup{ for all } k \leq {M-1}.
\end{equation}
Indeed, if for some $k \leq M-1$ the above inequality fails, then so does the $\tau$-concavity of $\tau \to \psi_\tau(x) + t \tau$, evaluated at the interior point $M/2^N$ of the interval $[k/2^{N} , (M+1)/2^{N}]$. This proves the claimed identity \eqref{eq: set_equal}. Moreover, \eqref{eq: many_est} and \eqref{eq: many_est_2} together imply that
\begin{equation}\label{eq: second_eq}
\check \psi_t^{N,M+1} = \psi_{(M+1)/2^N} + t (M+1)/2^{N}\,,\quad \check \psi_t^{N,M} = \psi_{M/2^N} + t M /2^{N}\, \ \  \textup{ on } U_t. 
\end{equation}
Since $U_t$ is an open set in the plurifine topology, we have $\theta^n_{\psi_{(M+1)/2^N}}\big|_{U_t}=\theta^n_{\check \psi^{N,M+1}_t}\big|_{U_t}$ and $\theta^n_{\psi_{M/2^N}}\big|_{U_t}=\theta^n_{\check \psi^{N,M}_t}\big|_{U_t}$ \cite{BT87} (see \cite[Lemma 2.1]{DDNL5}). 

Recall that the measure $\theta^n_{\psi_{(M+1)/2^N}}$ is supported on the set $\{\psi_{(M+1)/2^N} =0\}$, by \cite[Theorem~3.8]{DDNL2}.  Since  $\{\psi_{(M+1)/2^N} =0\} \subseteq U_t$ and $\{\psi_{(M+1)/2^N} =0\} \subset \{\psi_{M/2^N} =0\}$, the first inequality of \eqref{eq: first_I_ineq} implies the first inequality of \eqref{eq: diff_eq_I}:
\begin{flalign*}
\frac{t}{2^N} \int_X \theta^n_{\psi_{(M+1)/2^N}} & =\frac{t}{2^N} \int_{\{\psi_{(M+1)/2^N} =0\}} \theta^n_{\psi_{(M+1)/2^N}} \leq \int_{U_t}\left(\check{\psi}_t^{N,M+1}-\check{\psi}_t^{N,M}\right)\,\theta^n_{\check \psi_t^{N,M+1}} \\
& \leq I_\theta(\check \psi_t^{N,M+1})-I_\theta(\check \psi_t^{N,M}).
\end{flalign*}
We also have that
$$
\int_X\bigg(\check \psi_t^{N,M+1}-\check \psi_t^{N,M}\bigg)\theta^n_{\check \psi_t^{N,M}}=\int_{U_t}(\psi_{(M+1)/2^N}-\psi_{M/2^N}+\frac{t}{2^N})\theta^n_{\psi_{M/2^N}}\leq\frac{t}{2^N}\int_X\theta^n_{\psi_{M/2^N}}
$$
since $\psi_{(M+1)/2^N}\leq\psi_{M/2^N}$, which gives the second inequality of \eqref{eq: diff_eq_I}.

Fixing $N$, let $M = \lfloor 2^{N}\tau^{-}_\psi \rfloor \in \mathbb{Z}$. Repeated application of \eqref{eq: diff_eq_I} gives us
\[
\sum_{M+1 \leq j \leq 0} \frac{t}{2^N} \int_X \theta^n_{\psi_{j/2^N}}\leq I_\theta(\check \psi_t^{N,0})-I_\theta(\check \psi_t^{N,M})\leq \sum_{M \leq j \leq -1} \frac{t}{2^N} \int_X \theta^n_{\psi_{j/2^N}}\,.
\]
Since $M \leq 2^{N}\tau^-_\psi$ we have that $\check \psi^{N,M}_t = \psi_{M/2^N} + t M/2^N=V_\theta + t M/2^N$. As a result $I_\theta(\check \psi^{N,M}_t) = t M/2^N$, and we obtain
\[
\sum_{j=M+1}^0 \frac{t}{2^N} \left(\int_X \theta^n_{\psi_{j/2^N}} - \int_X \theta_{V_\theta}^n\right)\leq I_\theta(\check \psi_t^{N,0})\leq \sum_{j=M}^{-1} \frac{t}{2^N} \left(\int_X \theta^n_{\psi_{j/2^N}} - \int_X \theta_{V_\theta}^n\right)\,.
\]
We have Riemann sums on both the left and right of the above inequality. Using Lemma~\ref{lma:contimass} below, and the fact that $\check \psi^{N,0} \nearrow \check \psi_t$ a.e., it is possible to let $N \to \infty$ and obtain \eqref{eq: I_bounded_formula_RWN}, as desired.
\end{proof}

\begin{lemma}\label{lma:contimass}
Suppose that $\{\psi_\tau \}_{\tau}$ is a psh test curve. Then $\tau \mapsto \int_X \theta_{\psi_\tau}^n>0$ is a continuous function for $\tau \in (-\infty, \tau^+_\psi)$.
\end{lemma}

The proof of this result will be omitted, as it is exactly the same as \cite[Lemma 3.9]{DX20}, that deals with the K\"ahler case.

\medskip 

Lastly, we recall the \emph{maximization} $\{v_t\}_t$ of a finite energy sublinear subgeodesic ray $\{u_t\}_t$, which goes back to \cite[Proposition 4.6]{DDNL3} and is determined by the formula:
\begin{equation}\label{eq: maximization_def}
\hat v_\tau:= P[\hat u_\tau], \ \ \ \tau < \tau^+_{\hat u},
\end{equation}
 $\hat v_{\tau^+_{\hat u}} = \lim_{\tau \nearrow \tau^+_u } \hat v_\tau$ and $\hat v_\tau = -\infty $ for $\tau > \tau^+_{\hat u}$.
By the lemma below $\{\hat v_\tau\}_\tau$ is a maximal finite energy test curve and by Theorem \ref{thm: max_test_curve_ray_duality} we immediately have that $\{v_t\}_t$ is the smallest geodesic ray, satisfying $v_t \geq u_t, \ t > 0$.

\begin{lemma} $\{\hat v_\tau\}_\tau$ constructed above is a maximal psh test curve.
\end{lemma}

\begin{proof} 
The argument is very similar to  \cite[Proposition 4.6]{DDNL3}, so we will be again brief. By the definition of $P[\cdot]$ we see that $\tau \to \hat v_\tau$ is $\tau$-concave. By the proof of \cite[Lemma 4.3]{DDNL3} we see that $\tau \to v_\tau(x)$ is $\tau$-usc for any $x \in X$.

It remains to argue the maximality of $\{\hat v_\tau\}_\tau$. To start, we fix $\tau_1 \in (-\infty,\tau^+_{\hat u})$.
Since $v_\tau \nearrow V_\theta$ as $\tau \to -\infty$, by \cite[Theorem 2.1]{DDNL2} we have that $\int_X \theta_{v_{\tau_0}}^n > 0$ for some $\tau_0 < \tau_1$.
There exists $\alpha \in (0,1)$ such that $\tau_1 = \alpha \tau_0 + (1-\alpha) \tau^+_{\hat u}$. By $\tau$-concavity, we have that
$$\hat v_{\tau_1} \geq \alpha \hat v_{\tau_0} + (1-\alpha) \hat v_{\tau^+_{\hat u}}.$$
By \cite[Theorem 1.2]{WN19} and multilinearity of non-pluripolar products we obtain 
$\int_X \theta_{\hat v_{\tau_1}}^n \geq \alpha^n \int_X \theta_{\hat v_{\tau_0}}^n > 0$. As a result, we can apply \cite[Lemma 4.7]{DDNL3} to conclude that $\hat v_{\tau_1} = P[\hat v_{\tau_1}]$.
Lastly, we address maximality in the case $\tau := \tau^+_{\hat u}$. If $s < \tau=\tau^+_{\hat u}$, then by the above we can write
$P[\hat v_\tau] \leq P[\hat v_s] = \hat v_s.$ 
Letting $s \nearrow \tau = \tau^+_{\hat u}$,  we obtain that $P[\hat v_\tau] \leq \hat v_\tau$. Since the reverse inequality is trivial, we get $P[v_\tau] = v_\tau$, finishing the proof.
\end{proof}

Finally, we note the following identity for the radial  Monge--Amp\`ere energy (recall \eqref{eq: I-rad-subgeodesic}) of a finite energy subgeodesic and its maximization. It implies that the Legendre transform $\{\hat u_\tau\}_\tau$ of a finite energy sublinear subgeodesic ray $\{u_t\}_t$ is a finite energy test curve (recall Definition \ref{def: test_curves}(ii)). The reverse correspondence however might not be true.

\begin{proposition}\label{prop: I-rad_identity} Let $\{u_t\}_t$ be a finite energy sublinear subgeodesic ray and $\{v_t\}_t$ be its maximization constructed in \eqref{eq: maximization_def}. Then we have
\begin{equation}
\label{eq:radial-I-formula-subgeo-rays}
    I_\theta\{v_t\} = I_\theta\{u_t\} = \frac{1}{\vol(\{\theta\})}\int_{-\infty}^{\tau^+_{\hat u}} \left(\int_X \theta_{\hat u_\tau}^n-\int_X \theta_{V_\theta}^n \right) \,\mathrm{d}\tau+\tau_{\hat u}^+.
\end{equation}
\end{proposition}

\begin{proof} That $I_\theta\{v_t\} = \frac{1}{\vol(\{\theta\})}\int_{-\infty}^{\tau^+_{\hat u}} \left(\int_X \theta_{\hat u_\tau}^n-\int_X \theta_{V_\theta}^n \right) \,\mathrm{d}\tau+\tau_{\hat u}^+$ follows from \eqref{eq: I_RWN_form} and the fact that $\int_X \theta_{\hat v_\tau}^n = \int_X \theta_{P[\hat u_\tau]}^n = \int_X \theta_{\hat u_\tau}^n$ (\cite[Proposition 2.3 and Remark 2.5]{DDNL2}).

To finish the proof, we utilize an alternative construction for the maximization $\{v_t\}_t$. Namely, for $k \in \mathbb{N}$, let $\{v^k_t\}_t$ be the sublinear subgeodesic ray such that $v^k_t = u_t$ for $t \geq k$ and $[0,k] \ni t \to v^k_t \in \mathcal E^1(X,\theta)$ is the finite energy geodesic segment joining $V_\theta$ and $u_k$.
By the comparison principle we obtain that $\{v^k_t\}_t$ is indeed a sublinear subgeodesic ray, moreover $v_t\geq v^k_t \geq u_t$. Since the $k$-limit of $\{v^k_t\}_t$ is the smallest finite energy geodesic ray dominating $\{u_t\}_t$, we obtain that $v^k_t \nearrow v_t$.
Hence 
$I_\theta\{v_t\}\geq I_\theta\{u_t\}=\lim_{k}I_\theta\{v^k_t\}\geq\lim_{k}I_\theta(v^k_1)=I_\theta(v_1)=I_\theta\{v_t\}$,
where we used $I_\theta\{u_t\}=I_\theta\{v^k_t\}$ for all $k$,
$t \mapsto I_\theta(v^k_t)$ is convex and $t \mapsto I_\theta(v_t)$ is linear. So  $I_\theta\{v_t\}=I_\theta\{u_t\}$, finishing the proof.
\end{proof}

\section{Delta invariant and geodesic semistability}
\label{sec:delta=geodesic-stability}

In this section we prove Theorem \ref{mthm: main_radial_ding} and Theorem \ref{thm:delta=sup-D>0}. Let $\mu:=e^{\chi-\psi}$ be the tame measure defined in \eqref{eq:def-mu} and fix $\lambda\in(0,c_\mu[V_\theta])$. Here we recall that $\chi, \psi$ are qpsh functions on $X$ with $\chi$ assumed to have analytic singularity type.

Our first result gives a precise formula for the slope of the $\lambda$-Ding functional (see \eqref{def: twisted_Ding}) along subgeodesic rays, rooted in the ideas of \cite[Section 4]{DX20}:

\begin{theorem}[=Theorem \ref{mthm: main_radial_ding}]
\label{thm: radial_ding} Let $\{u_t\}_t \subset \mathcal E^1(X,\theta)$ be a sublinear subgeodesic ray. Then
\begin{flalign*}
\liminf_{t \to \infty} \frac{\mathcal  D_\mu^\lambda(u_t)}{t} = -I_\theta\{u_t\} + \sup\{\tau  :  \int_X e^{-\lambda \hat u_\tau } d \mu <\infty\} = -I_\theta\{u_t\} + \sup\{\tau  :  c_{\mu}[\hat u_\tau] \geq \lambda\}.
\end{flalign*} 
\end{theorem}

\begin{proof}
By basic properties of the Legendre transform, $\frac{\sup_X (u_t - V_\theta)}{t}  \nearrow \tau^+_{\hat u}$, as $t \to \infty$. Hence, after adding a $t$-linear term to $u_t$ we can assume that $\frac{\sup_X (u_t - V_\theta)}{t} \nearrow \tau^+_{\hat u} = 0$. In particular, $t \mapsto u_t$ is $t$-decreasing.

Comparing \eqref{eq: I-rad-subgeodesic} and the definition $\mathcal D^\lambda_\mu$, it is enough to investigate the following radial functional:
\begin{equation}
    \label{eq:def-L}
    \mathcal L^\lambda_\mu\{u_t\}:=\liminf_{t \to \infty} \frac{-1}{\lambda t }\log \int_X e^{-\lambda u_t} \mathrm{d}\mu. 
\end{equation}
It amounts to showing that
\
$$\mathcal L^\lambda_\mu\{u_t\}= \tau_D := \sup\{\tau \ : \  \int_X e^{-\lambda \hat u_\tau} \mathrm{d}\mu <\infty\}.$$ 
We note that $\tau_D>-\infty$ by Theorem \ref{thm: openness}, since $\hat{u}_\tau\nearrow V_\theta$ in $L^1$ as $\tau\rightarrow-\infty$ and $\int_Xe^{-\lambda V_\theta}\mathrm{d}\mu<\infty$ by our assumption on $\lambda$.

Let $\tau<\tau_D$ and $C:=\int_X e^{-\lambda \hat{u}_{\tau}}\, d\mu<\infty$. For any $t\geq 0$ we have $\hat{u}_{\tau}\leq u_t-\tau t$, so $C\geq \int_X e^{-\lambda(u_t-\tau t)}\,\mathrm{d}\mu$. As a result, $\mathcal L^\lambda_\mu\{u_t\} \geq \tau$, hence
\[
\mathcal L^\lambda_\mu\{u_t\} \geq \tau_D\,.
\]
Now we prove the reverse inequality. We fix $p< L^\lambda_\mu\{u_t\}$ and $\epsilon >0$ satisfying $p+\varepsilon <  \mathcal L^\lambda_\mu\{u_t\}$. We can find $t_0>0$ such that 
\[
 \int_X e^{-\lambda u_t} d\mu <e^{-(p+\varepsilon)\lambda t}\,,\quad t \geq t_0\,.
\]
Hence $\int_0^{\infty}e^{p\lambda t}\int_X e^{-\lambda u_t} d\mu \,\mathrm{d}t<\infty$.
By Fubini--Tonelli this is equivalent to
\begin{equation}\label{eq:finiteness}
\int_X \int_0^{\infty} e^{\lambda (p t-u_t)} \mathrm{d}t\,\mathrm{d}\mu<\infty\,.
\end{equation}

Before proceeding further, we notice that the following estimate holds, using the fact that $u_t \leq V_\theta + t \sup_X (u_t - V_\theta) \leq V_\theta \leq 0$:
\[
p < \mathcal L^\lambda_\mu\{u_t\} \leq 0. \]
As a result, $\hat  u_{p} = \inf (u_t - p t)$ is not identically equal to $-\infty$, since $\tau_u^+=0$. Let $x \in X$ such that $\hat u_{p}(x)$, $\chi(x)$ and $\psi(x)$ are finite. 
By definition of $\hat{u}_{p}$, we can find $t_0>0$ so that $\hat u_{p}(x)+1 \geq  u_{t_0}(x)-p t_0$.
Since $t \mapsto u_t(x)$ is decreasing,  we have $\hat{u}_{p}(x)-p+1\geq u_{t}(x)-p t$ for $t\in [t_0,t_0+1]$. 
Hence
\begin{flalign*}
\int_0^{\infty}e^{\lambda (pt - u_t(x))+\chi(x)-\psi(x)}\,\mathrm{d}t&\geq \int_{t_0}^{t_0 + 1}e^{\lambda(pt-u_t(x))+\chi(x)-\psi(x)}\,\mathrm{d}t\\
&\geq \int_{t_0}^{t_0+1}e^{-\lambda \hat u_p(x)}e^{\lambda (p-1)+\chi(x)-\psi(x)}\,\mathrm{d} t\geq e^{\lambda(p-1)}e^{-\lambda \hat{u}_{p}(x)+\chi(x)-\psi(x)}\,.
\end{flalign*}
By this and \eqref{eq:finiteness} we obtain that $\int_X e^{-\lambda \hat u_p} d\mu < \infty$, 
hence $p\leq \tau_D$, concluding the proof.
\end{proof}

Motivated by the above result, we introduce  
$$\mathcal D_\mu^\lambda \{u_t\} := \liminf_{t \to \infty} \frac{\mathcal  D_\mu^\lambda(u_t)}{t},$$  
and we will call this expression the radial $\lambda$-Ding functional of the finite energy sublinear subgeodesic ray $\{u_t\}_t$.

\begin{remark}
In the Fano case, when the geodesic ray $\{u_t\}_t$ is induced by a filtration of the section ring, $\mathcal{D}^\lambda\{u_t\}$ actually coincides with the $\beta_\lambda$-invariant introduced in \cite[\S 4]{XZ20}.
\end{remark}

Next, we move on to prove Theorem \ref{thm:delta=sup-D>0}.
We start with the following preliminary estimate regarding the $\delta_\mu$ invariant defined in \eqref{eq: delta_mu_def}:
\begin{lemma}\label{lem: c_delta_est}We have
$c_\mu[V_\theta]\geq\delta_\mu$.
\end{lemma}
\begin{proof} Let $\pi: Y \to X$ be a smooth bimeromorphic model over $X$ and $E$ be a prime divisor in $Y$. Recall that $\tau_\theta(E)  = \sup \{\tau>0  \textup{ s.t. } \{\pi^*\theta\} - \tau\{E\} \textup{ is a big class}\}$ is the pseudo-effective threshold of $E$. Equivalently, one has
\begin{equation*}\label{eq: tau_theta_formula}
\tau_\theta(E)=\sup_{u\in\textup{PSH}(X,\theta)}\nu(u,E).
\end{equation*}
So in particular, $\nu(V_\theta,E)\leq\tau_\theta(E)$. A  well-known fact that we shall use is (see \cite{Bou04})
$$
\vol(\{\pi^*\theta\}-t\{E\})=\vol(\{\theta\}),\ t\in[0,\nu(V_\theta,E)].
$$

We reproduce the proof of this identity for the reader's convenience.
Let $h_E$ be a smooth hermitian metric on $\mathcal O_Y(E)$, with canonical section $S_E$ and $\Theta(h_E):=-\ddc\log h_E$. We notice that 
$$[V_{\pi^*\theta  - t \Theta(h_E)}] = [V_{\pi^*\theta} - t \log |S_E|_{h_E}^2]=[V_{\theta} \circ \pi - t \log |S_E|_{h_E}^2], \ \ t \in [0, \nu(V_\theta, E)].$$
This follows from the fact that the map $U \mapsto U - t \log |S_E|_{h_E}^2$ gives a monotone increasing bijection between $\textup{PSH}(Y,\pi^*\theta)$ and $\textup{PSH}(Y,\pi^*\theta - t \Theta(h_E))$ for $t \in [0, \nu(V_\theta, E)]$, hence preserves potentials of minimal singularity type.

As a result, using \cite[Theorem 1.2]{WN19}, for $t \in [0, \nu(V_\theta, E)]$ we get that
\begin{flalign*}
\vol(\{\pi^*\theta\} - t \{E\}) &= \int_Y ( \pi^* \theta - t\Theta(h_E) + \ddc V_{\pi^*\theta  - t \Theta(h_E)})^n \\
&= \int_Y ( \pi^* \theta - t\Theta(h_E) + \ddc (V_{\pi^*\theta} - t \log |S_E|_{h_E}^2))^n= \int_X 
\theta_{V_\theta}^n,
\end{flalign*}
where in the last line we used $\Theta(h_E) + \ddc \log |S_E|^2_{h_E} = 0$ on $Y \setminus E$. Then from $\nu(V_\theta, E) \leq \tau_\theta(E)$ and the above identity we obtain that  
\begin{equation}\label{eq: S_nu_ineq}
S_\theta(E):=\frac{1}{\vol(\{\theta\})}\int_0^{\tau_\theta(E)}\vol(\{\pi^*\theta\}-x\{E\})dx \geq  \nu(V_\theta, E).
\end{equation}
Using Theorem \ref{thm:c=inf-A/S}, this allows to conclude:
$c_\mu[V_\theta] = \inf_E \frac{A_{\chi,\psi}(E)}{\nu(V_\theta,E)} \geq \inf_E \frac{A_{\chi,\psi}(E)}{S_\theta(E)} = \delta_\mu.$
\end{proof}

Let $E\subset Y\xrightarrow{\pi}X$ be a prime divisor over $X$. We shall make use of the following bounded test curve naturally associated to $E$, going back to \cite{RWN14}:
\begin{equation}
\label{eq:psi=P[u]}
 \hat u^E_\tau:=
    \begin{cases}
  V_\theta,\ &\tau\leq 0;\\
  v^E_\tau,\ &0 < \tau<\tau_\theta(E);\\
  \lim_{\tau\nearrow\tau_\theta(E)}v^E_\tau,\ &\tau=\tau_\theta(E);\\
  -\infty,\ &\tau>\tau_\theta(E),
\end{cases}
\end{equation}
where $v^E_\tau$ is given by
$ v^E_\tau:=\sup\big\{\varphi\in\PSH(X, \theta)\big| \nu(\varphi,E) \geq \tau, \  
\varphi \leq 0 \big\}.
$
Using basic properties of Lelong numbers, one can show that $v^E_\tau\in\textup{PSH}(X,\theta)$ and that
\begin{equation}
\label{eq:nu-v-E-tau=tau}
    \nu(v^E_\tau,E)\geq\tau.
\end{equation}
The test curve $\{\hat u^E_\tau\}_\tau$ is actually maximal, but we do not need this in what follows.

Let $\{u^E_t\}_t$ be the finite energy (sub)geodesic ray associated to the test curve $\{\hat u^E_\tau\}_\tau$.
We then have the following observation.
\begin{lemma} 
\label{lem:I=S}
We have $I_\theta\{u^E_t\} = S_\theta (E)$.
\end{lemma}
\begin{proof}
Let $h_E$ be a smooth hermitian metric on $\mathcal O(E)$, with canonical section $S_E$ and $\Theta(h_E):=-\ddc\log h_E$. Using \eqref{eq:nu-v-E-tau=tau} we notice the following identity of singularity types:
$$[\pi^*v^\tau_E -\tau \log|S_E|^2_{h_E}] = [V_{\pi^*\theta - \tau \Theta(h_E)}], \ \ \tau \in [0, \tau_\theta(E)].$$
This follows from the fact that the map $U \mapsto U - \tau \log |S_E|_{h_E}^2$ gives a monotone increasing bijection between $\{w \in \textup{PSH}(Y,\pi^*\theta), \ \ \nu(w,E) \geq \tau\}$ and $\textup{PSH}(Y,\pi^*\theta - \tau \Theta(h_E))$, hence preserves qpsh potentials of minimal singularity type in each set.

Since $\Theta(h_E) + \ddc \log|S_E|^2_{h_E} = 0$ on $Y \setminus E$, we have that 
$$(\pi^*\theta - \tau \Theta(h_E) + \ddc (\pi^*v_\tau^E -\tau \log|S_E|^2_{h_E}))^n  = (\pi^*\theta + \ddc \pi^* v_\tau^E )^n.$$
As a result, using \cite[Theorem 1.2]{WN19}, we notice that for $\tau < \tau_\theta(E)$ we have 
\begin{equation}
    \label{eq:int-u=vol(theta-tau-E)}
    \int_X \theta^n_{\hat u^E_\tau}= \int_Y \pi^*\theta^n_{\pi^*v^E_\tau} = \vol(\{\pi^* \theta\} - \tau \{E\}).
\end{equation}
That $I_\theta\{u^E_t\} = S_\theta (E)$, follows after comparing  \eqref{eq: SE_def} and \eqref{eq:radial-I-formula-subgeo-rays}.
\end{proof}

Next we introduce a new valuative invariant attached to certain test curves.
Given a sublinear subgeodesic ray $\{u_t\}_t\subset\mathcal{E}^1(X,\theta)$, it follows from \eqref{eq:radial-I-formula-subgeo-rays} that the Legendre transform $\{\hat u_\tau\}_\tau$ of $\{u_t\}_t$ is a finite energy test curve. For any prime divisor $E$ over $X$, we define \emph{the expected Lelong number} $\nu(\{\hat u_\tau\}_\tau,E)$ of the test curve $\{\hat u_\tau\}_\tau$ along $E$ to be
\begin{equation*}
\label{eq:def-Z-psi}
\nu(\{\hat u_\tau\}_\tau,E):=
    \begin{cases}
    \nu(\hat u_{I_\theta\{u_t\}},E), &\text{ if }I_\theta\{u_t\}<\tau^+_{\hat u},\\
    \nu(V_\theta,E), &\text{ if }I_\theta\{u_t\}=\tau^+_{\hat u}.\\
    \end{cases}
\end{equation*}

%When $I_\theta\{\check \psi_t\}=\tau^+_\psi$, one must have $\int_X\theta^n_{\psi_\tau}=\vol(\{\theta\})$ for all $\tau\in(-\infty,\tau^+_\psi)$, i.e., $\psi_\tau$ has full mass. So by \cite[Theorem 1.1]{DDNL1} and \cite{BFJ08} one has $\nu(\psi_\tau,E)=\nu(V_\theta,E)$ for all $\tau\in(-\infty,\tau^+_\psi)$. For instance one may consider the trivial test curve defined by $\psi_\tau:=V_\theta+1/\tau$ for $\tau\in(-\infty,0)$ and $\psi_\tau:=-\infty$ for $\tau\geq0$, which explains why we put $\nu(\{\psi_\tau\},E)$ to be $\nu(V_\theta,E)$ in this case.

The following estimate for the expected Lelong number will play a key role, and should be compared with \cite[Lemma 2.2]{FO18}.
\begin{proposition}
\label{prop:nu-psi-E<S-E}
For any sublinear subgeodesic ray $\{u_t\}_t\subset\mathcal{E}^1(X,\theta)$ and any prime divisor $E$ over $X$ we have
$$
\nu(\{\hat u_\tau\}_\tau,E)\leq S_\theta(E).
$$
\end{proposition}

\begin{proof}
Put $f(\tau):=\nu(\hat u_\tau,E)$ for $\tau\in(-\infty,\tau^+_{\hat u})$.
Then it is clear that $f(\tau)$ is a non-negative convex non-decreasing function defined on $(-\infty,\tau^+_{\hat u})$. So in particular we may set
$
f(-\infty):=\lim_{\tau\searrow-\infty}f(\tau).
$
When $I_\theta\{u_t\}=\tau^+_{\hat u}$, the desired inequality is $\nu(\{\hat u_\tau\},E)=\nu(V_\theta,E)\leq S_\theta(E)$, which is  \eqref{eq: S_nu_ineq}. 

So in what follows we can assume that $I_\theta\{u_t\}<\tau^+_{\hat u}$. We first consider the case when $f(-\infty)=f(I_\theta\{u_t\})$. This means that $f(\tau)=a \in \mathbb{R}$, for $\tau\in(-\infty,I_\theta\{u_t\}]$. Namely, $\nu(\hat u_\tau,E)\equiv a$ for $\tau\in(-\infty,I_\theta\{u_t\}]$. So by \cite[Theorem 1.2]{WN19} and \eqref{eq:int-u=vol(theta-tau-E)} we have
$$
\int_X\theta^n_{\hat u_\tau}\leq\int_X\theta^n_{v^E_{a}}=\vol(\{\pi^*\theta\}-a\{E\})\text{ for }\tau\in(-\infty,I_\theta\{u_t\}],
$$
where $v^E_{a}$ is defined below \eqref{eq:psi=P[u]}.
Thus, by \eqref{eq:radial-I-formula-subgeo-rays}
$$
I_\theta\{u_t\}\leq\frac{1}{\vol(\{\theta\})}\int_{-\infty}^{I_\theta\{u_t\}}\bigg(\vol(\{\pi^*\theta\}-a\{E\})-\vol(\{\theta\})\bigg)\mathrm{d}\tau+\tau^+_{\hat u},
$$
which forces that $\vol(\{\pi^*\theta\}-a\{E\})=\vol(\{\theta\})$ since $\{\hat u_\tau\}_\tau$ has finite energy.
So we have
$$
S_\theta(E)=\frac{1}{\vol(\{\theta\})}\int_0^{\tau_\theta(E)}\vol(\{\pi^*\theta\}-x\{E\})\mathrm{d}x\geq\frac{1}{\vol(\{\theta\})}\int_0^a\vol(\{\pi^*\theta\}-a\{E\})\mathrm{d}x=a,
$$
what we aimed to prove. Thus, we can assume that $f(-\infty)<f(I_\theta\{u_t\})$. Put
$$
b:=f^\prime_-(I_\theta\{u_t\}):=\lim_{h\rightarrow 0^+}\frac{f(I_\theta\{u_t\})-f(I_\theta\{u_t\}-h)}{h},
$$
which has to be a finite positive number by the convexity of $f$. Now we put
\begin{equation*}
    g(\tau):=
    \begin{cases}
    0,&\tau\in(-\infty,I_\theta\{u_t\}-b^{-1} f(I_\theta\{u_t\})),\\
    b(\tau-I_\theta\{u_t\})+f(I_\theta\{u_t\}),&\tau\in[I_\theta\{u_t\}-b^{-1} f(I_\theta\{u_t\}),\tau^+_{\hat u}].\\
    \end{cases}
\end{equation*}
Due to convexity, we have $f(\tau)\geq g(\tau), \ \tau\in(-\infty,\tau^+_{\hat u})$. By \cite[Theorem 1.2]{WN19} and \eqref{eq:int-u=vol(theta-tau-E)},
$$
\int_X\theta^n_{\hat u_\tau}\leq\int_X\theta^n_{v^E_{g(\tau)}}=\vol(\{\pi^*\theta\}-g(\tau)\{E\})\text{ for }\tau\in(-\infty,\tau^+_{\hat u}).
$$
Thus using \eqref{eq:radial-I-formula-subgeo-rays} we have that
\begin{equation*}
    \begin{aligned}
    I_\theta\{u_t\}&
     \leq\frac{1}{\vol(\{\theta\})}\int_{-\infty}^{\tau^+_{\hat u}}\bigg(\vol(\{\pi^*\theta\}-g(\tau)\{E\})-\vol(\{\theta\})\bigg)\mathrm{d}\tau+\tau^+_{\hat u}\\
    &\leq\frac{1}{\vol(\{\theta\})}\int_{I_\theta\{u_t\}-b^{-1} f(I_\theta\{u_t\})}^{\tau^+_{\hat u}}\bigg(\vol(\{\pi^*\theta\}-g(\tau)\{E\})-\vol(\{\theta\})\bigg)\mathrm{d}\tau+\tau^+_{\hat u}\\
    &\leq\frac{1}{\vol(\{\theta\})b}\int_0^{\tau_\theta(E)}\vol(\{\pi^*\theta\}-x\{E\})\mathrm{d}x+I_\theta\{u_t\}-b^{-1}f(I_\theta\{u_t\})\\
    &=b^{-1}(S_\theta(E)-f(I_\theta\{u_t\})+I_\theta\{u_t\},
    \end{aligned}
\end{equation*}
Using the above inequality, since $b >0,$ we finally arrive at
$
\nu(\{\hat u_\tau\}_\tau,E)=\nu(\hat u_{I_\theta\{u_t\}},E)=f(I_\theta\{u_t\})\leq S_\theta(E),
$
finishing the proof.
\end{proof}

Inspired by the above, for any sublinear subgeodesic ray $\{u_t\}_t\subset\mathcal{E}^1(X,\theta)$ we put
\begin{equation*}
    \label{def:eq-c-psi-tau}
    c_\mu\{\hat u_\tau\}:=
    \begin{cases}
    c_\mu[\hat u_{I_\theta\{u_t\}}], & \text{ if }I_\theta\{u_t\}<\tau^+_{\hat u},\\
    c_\mu[V_\theta], & \text{ if }I_\theta\{u_t\}=\tau^+_{\hat u},\\
    \end{cases}
\end{equation*}
which is called the \emph{expected complex singularity exponent} of $\{\hat u_\tau\}_\tau$.

We show the following  analytic characterization of the $\delta$-invariant that should be compared with Fujita--Odaka's basis-divisor formulation \cite{FO18}.

\begin{theorem}
\label{thm:delta=inf-c[psi]'}
We have
$$
\delta_\mu=\inf_{\{u_t\}}c_\mu\{\hat u_\tau\},
$$
where the inf is over all sublinear subgeodesic rays $\{u_t\}_t\subset\mathcal{E}^1(X,\theta)$.
\end{theorem}

This result implies that  $\delta_\mu>0$. Indeed,  we have
$
\delta_\mu\geq\inf_{u\in\textup{PSH}(X,\theta)}c_\mu[u].
$
By strong openness \cite{GuZh15} the tame measure $\mu$ has $L^p$ density for some $p>1$. So H\"older's inequality and \cite{DK01,Tia87} imply that
$
c_\mu[u]\geq\frac{p-1}{p} c_{\omega^n}[u] >\alpha
$ for some uniform $\alpha >0$.

\begin{proof}
Given $\{u_t\}_t$, if  
$
I_\theta\{u_t\}=\tau^+_{\hat u},
$
then by Lemma \ref{lem: c_delta_est}
$
c_\mu\{\hat u_\tau\}=c_\mu[V_\theta]\geq\delta_\mu
$
. If $I_\theta\{u_t\}<\tau^+_{\hat u}$, then
Theorem \ref{thm:c=inf-A/S} and Proposition \ref{prop:nu-psi-E<S-E} imply that
$$
c_\mu\{\hat u_\tau\}=c_\mu[\hat u_{I_\theta\{u_t\}}]=\inf_E\frac{A_{\chi,\psi}(E)}{\nu(\hat u_{I_\theta\{u_t\}},E)}\geq\inf_E\frac{A_{\chi,\psi}(E)}{S_\theta(E)}=\delta_\mu.
$$
So we obtain that
$
\inf_{\{u_t\}}c_\mu\{\hat u_\tau\}\geq\delta_\mu.
$

To see the reverse direction, let $E$ be any prime divisor over $X$ and let $\{\hat u^E_\tau\}_\tau$ be the bounded test curve constructed in \eqref{eq:psi=P[u]}. Then by Lemma \ref{lem:I=S} we have
$
I_\theta\{u^E_t\}=S_\theta(E),\text{ and } \tau^+_{\hat u^E}=\tau_\theta(E).
$
If $I_\theta\{u^E_t\}=\tau^+_{\hat u^E}$, i.e., $S_\theta(E)=\tau_\theta(E)$, then $\hat u_\tau^E\in\mathcal{E}(X,\theta)$ for all $\tau\in(-\infty,\tau_\theta(E))$ in light of \eqref{eq:radial-I-formula-subgeo-rays}. So \cite[Theorem 1.1]{DDNL1} and \eqref{eq:nu-v-E-tau=tau}
imply that $\nu(V_\theta,E)=\nu(\hat u^E_\tau,E)\geq\tau$ for all $\tau\in(-\infty,\tau_\theta(E)),$ and hence $\nu(V_\theta,E)=\tau_\theta(E)=S_\theta(E).$ So by Theorem \ref{thm:c=inf-A/S} again,
$$
c_\mu\{\hat u^E_\tau\}=c_\mu[V_\theta]\leq\frac{A_{\chi,\psi}(E)}{\nu(V_\theta,E)}=\frac{A_{\chi,\psi}(E)}{S_\theta(E)}.
$$
If $I_\theta\{u^E_t\}<\tau^+_{\hat u^E}$, i.e., $S_\theta(E)<\tau_\theta(E)$, then from \eqref{eq:nu-v-E-tau=tau} we also have $\nu(\hat u^E_{S_\theta(E)},E)\geq S_\theta(E)$ (this is actually an equality, in view of Proposition \ref{prop:nu-psi-E<S-E}). The we derive that 
$$
c_\mu\{\hat u^E_\tau\}=c_\mu[\hat u^E_{I_\theta\{u^E_t\}}]=c_\mu[\hat u^E_{S_\theta(E)}]\leq\frac{A_{\chi,\psi}(E)}{\nu(\hat u^E_{S_\theta(E)},E)}\leq\frac{A_{\chi,\psi}(E)}{S_\theta(E)},
$$
by Lemma \ref{lem:I=S} and Theorem \ref{thm:c=inf-A/S}. This implies the reverse direction.
\end{proof}

Now we are in the position to prove the following key result.

\begin{theorem}[=Theorem \ref{thm:delta=sup-D>0}]
\label{thm:delta=sup=rad-Ding>0'}
One has
\begin{equation}\label{eq: geod_stab_threshold_id}
\delta_\mu=\sup\{\lambda>0|\mathcal{D}^\lambda_\mu\{u_t\}\geq 0\text{ for all sublinear subgeodesic ray }\{u_t\}_t\subset\mathcal{E}^1(X,\theta)\}.
\end{equation}
\end{theorem}

\begin{proof}
Assume that $
\mathcal{D}^\lambda\{u_t\}\geq0$ holds for any sublinear subgeodesic ray $\{u_t\}_t \subset \mathcal E^1(X,\theta).
$
Consider the associated test curve $\{\hat{u}_\tau\}_\tau$.
Set for simplicity
$
\beta:=\sup\{\tau:\ \int_Xe^{-\lambda\hat{u}_\tau}\mathrm{d}\mu<\infty\}.
$
Then by Theorem \ref{thm: radial_ding},
$
\beta-I_\theta\{u_t\}=\mathcal{D}^\lambda_\theta\{u_t\}\geq0,
$
so that $\beta\geq I_\theta\{u_t\}$. This means that for any $\tau<I_\theta\{u_t\}$, $c_\mu[\hat u_\tau]\geq\lambda$. 

If $I_\theta\{\hat u_\tau\}=\tau^+_{\hat u}$, then
$c_\mu\{\hat u_\tau\}=c_\mu[V_\theta]\geq c_\mu[\hat u_\tau]\geq\lambda$. 
If $I_\theta\{\hat u_\tau\}<\tau^+_{\hat u}$, then by Lemma \ref{lem:1/c-convex} we also have $c_\mu\{\hat u_\tau\}=c_\mu[\hat u_{I_\theta\{u_t\}}]\geq\lambda$. Taking inf over all such $\{u_t\}_t$ then yields $\delta_\mu\geq\lambda$ by the previous theorem. Hence, 
the right hand side of \eqref{eq: geod_stab_threshold_id} is dominated by $\delta_\mu$.

It remains to argue that for any $\lambda\in(0,\delta_\mu)$, 
$
\mathcal{D}^\lambda_\mu\{u_t\}\geq0$ holds for any sublinear subgeodesic ray $\{u_t\}_t \subset \mathcal E^1(X,\theta)$. To start we note by Proposition \ref{prop:c-u=c-P-u} and Lemma \ref{lem: c_delta_est} that $\mathcal D^\lambda_\mu(v)$ is finite for any $v \in \mathcal E^1(X,\theta)$.

Consider the associated test curve $\{\hat{u}_\tau\}_\tau$ and put again
$
\beta:=\sup\{\tau:\ \int_Xe^{-\lambda\hat{u}_\tau}\mathrm{d}\mu<\infty\}.
$
First, if 
$I_\theta\{u_t\}=\tau^+_{\hat u}$, then in view of \eqref{eq:radial-I-formula-subgeo-rays} $\hat u_\tau$ has full mass for any $\tau<\tau^+_{\hat u}$, so that
$
\int_Xe^{-\lambda\hat u_\tau}\mathrm{d}\mu<\infty\text{ for all }\tau\in(-\infty,\tau^+_{\hat u})
$
by Proposition \ref{prop:c-u=c-P-u} and Lemma \ref{lem: c_delta_est}. Thus by Theorem \ref{thm: radial_ding},
$
\mathcal{D}_\mu^\lambda\{u_t\}=\beta-I_\theta\{u_t\}=\tau^+_{\hat u}-\tau^+_{\hat u}=0,
$
as desired.
Next, if $I_\theta\{u_t\}<\tau^+_{\hat{u}}$, then we can apply Theorem \ref{thm:delta=inf-c[psi]'} to conclude that
$
c_\mu[\hat u_{I_\theta\{u_t\}}]\geq\delta_\mu>\lambda,
$
so that
$
\beta\geq I_\theta\{u_t\}
$.
Thus
$
\mathcal{D}_\mu^\lambda\{u_t\}=\beta-I_\theta\{u_t\}\geq I_\theta\{u_t\}-I_\theta\{u_t\}=0,
$
finishing the proof.
\end{proof}

%As a consequence of Theorem \ref{thm:delta=sup=rad-Ding>0'} we note the following practical  result.

\begin{proposition}
\label{prop:delta>1=>geodesic-stable}
For any $\lambda\in(0,\delta_\mu)$, then there exists $\varepsilon>0$ such that
\begin{equation}\label{eq: D_rad_est}
\mathcal{D}^\lambda_\mu\{w_t\}\geq-\varepsilon I_\theta\{w_t\}.
\end{equation}
for all finite energy geodesic rays $\{w_t\}_t\in\mathcal{R}^1(X,\theta)$ with $\sup_X w_t=0$. In particular, if $\{w_t\}_t$ is also $d_1$-unit speed, then $
\mathcal{D}^\lambda_\mu\{w_t\}\geq \varepsilon$.
\end{proposition}

\begin{proof}
Consider Tian's $\alpha$-invariant \cite{Tia87} adapted to our setting:
\begin{equation}
    \label{eq:def-alpha}
    \alpha_\mu:=\sup\bigg\{\alpha>0:\sup_{u\in\textup{PSH}(X,\theta)}\int_Xe^{-\alpha(u-\sup_X u)}\mathrm{d}\mu<\infty\bigg\},
\end{equation}
which is positive. Indeed, pick $A>0$ sufficiently large such that $\theta+A\omega$ is K\"ahler, so that $\textup{PSH}(X,\theta)\subset\textup{PSH}(X,\theta+A\omega)$. The positivity of $\alpha_\mu$ follows from \cite[Proposition 1.4]{BBEGZ16}. 

Fix $\eta\in(\lambda,\delta_\mu)$ and $\alpha\in(0,\min\{\lambda,\alpha_\mu\})$. By H\"older's inequality, for all $w \in \mathcal E^1(X,\theta)$ we have
\begin{equation*}
	\begin{aligned}
		-\frac{1}{\lambda}\log\int_Xe^{-\lambda w}\mathrm{d}\mu &\geq-\frac{\lambda-\alpha}{\lambda(\eta-\alpha)}\log\int_Xe^{-\eta w}\mathrm{d}\mu-\frac{\eta-\lambda}{\lambda(\eta-\alpha)}\log \int_Xe^{-\alpha w}\mathrm{d}\mu.\\
	\end{aligned}
\end{equation*}
Therefore, one can find $C>0$ such that for all $w \in \mathcal E^1(X,\theta)$ with $\sup_X w = 0$ we have
\begin{equation} \label{eq: D_twist_ineq}
{\mathcal{D}^\lambda_\mu(w)}\geq\frac{\eta(\lambda-\alpha)}{\lambda(\eta-\alpha)} \mathcal{D}^\eta_\mu(w)-\frac{\alpha(\eta-\lambda)}{\lambda(\eta-\alpha)} I_\theta(w) - C.
\end{equation}
Putting $w_t$ in the above inequality, dividing with $t$ and letting $t\rightarrow\infty$, we find that
$$
\mathcal{D}^\lambda_\mu\{w_t\}\geq\frac{\eta(\lambda-\alpha)}{\lambda(\eta-\alpha)}\mathcal{D}^\eta_\mu\{w_t\}-\frac{\alpha(\eta-\lambda)}{\lambda(\eta-\alpha)}I_\theta\{w_t\}.
$$
So we conclude \eqref{eq: D_rad_est} by Theorem \ref{thm:delta=sup=rad-Ding>0'}. In case $\{w_t\}_t$ is unit speed, we have that $1 = d_1(V_\theta,w_1) = -I_\theta(w_1).$ By linearity of $t \mapsto I_\theta(w_t)$ the last statement also follows.
\end{proof}

\section{Twisted K\"ahler--Einstein metrics}
\label{sec:big-YTD}

\subsection{Energy properness of the Ding functionals}

As in the previous section, let $\{\theta\}$ be a big class and $\mu$ be a tame measure defined in \eqref{eq:def-mu}.
We will be interested in the following KE type complex Monge--Amp\`ere equation: 
\begin{equation}
    \label{eq:tKE-MA-equation}
    \theta_u^n=e^{-\lambda u}\mathrm{d}\mu
\end{equation}
for $\lambda>0$. To make sense of the above equation, we need to further assume that 
$$
\lambda\in(0,c_{\mu}[V_\theta]).
$$
Due to Proposition \ref{prop:c-u=c-P-u} we have $c_\mu[v] > \lambda$ for any $v \in \mathcal E^1(X,\theta)$, hence the $\lambda$-twisted Ding energy $\mathcal{D}^\lambda_\mu$ corresponding to \eqref{eq:tKE-MA-equation} can be defined as follows, having an exponential part $\mathcal L^\lambda_\mu$ and an energy part:
\begin{equation}\label{eq: D_L_def}
\mathcal{D}^\lambda_\mu(u):=\mathcal L_\mu^\lambda(u) - I_\theta(u)= -\frac{1}{\lambda}\log \int_Xe^{-\lambda u}\mathrm{d}\mu-I_\theta(u)\text{ for }u\in\mathcal{E}^1(X,\theta).
\end{equation}
\begin{definition}
    We say that $\mathcal{D}^\lambda_\mu$ is coercive/proper if there exists $\varepsilon>0$ and $C>0$,
$$
\mathcal{D}^\lambda_\mu(u)\geq\varepsilon(\sup_X u-I_\theta(u))-C, \ \ \ u\in\mathcal{E}^1(X,\theta).$$
\end{definition}

 As is well known (cf. \cite{BBGZ13,BBEGZ16}), properness of the twisted Ding energy $\mathcal D^\lambda_\mu$ implies that  \eqref{eq:tKE-MA-equation} admits a finite energy solution:

\begin{proposition}\label{prop: proper_implies_minimizer}
	If $\mathcal D^\lambda_\mu$ is proper, and $\mathcal L_\mu^\lambda$ is $L^1$-continuous, then $\mathcal D^\lambda_\mu$ has a global minimizer in $\mathcal E^1(X,\theta)$. Any such minimizer is a solution to \eqref{eq:tKE-MA-equation}, up to an additive constant.
\end{proposition}

We believe that the $L^1$-continuity of $\mathcal L_\mu^\lambda$ holds unconditionally. However to stay focused, we do not attempt to prove this here. We do point out later that this always holds in case of our applications, as an easy consequence of \cite[Theorem 0.2(2)]{DK01}. 

\begin{proof}
	Let $u_i\in\mathcal E^1(X,\theta)$ with $\sup_X u_i=0$ be a minimizing sequence:
	$$
	\lim_{i\rightarrow\infty}\mathcal D^\lambda_\mu(u_i)=\inf_{u\in\mathcal E^1(X,\theta)}\mathcal D^\lambda_\mu(u).
	$$
	The properness condition implies that we can find $C>0$ such that $I_\theta(u_i)\geq -C$ for all $i$.
	Then up to a subsequence one can find $u_\infty\in\mathcal E^1(X,\theta)$ such that $u_i\rightarrow u_\infty$ in $L^1$. Note that by our assumption $\mathcal D^\lambda_\mu$ is $L^1$-lsc on $\mathcal{E}^1(X,\theta)$.
	So $u_\infty$ is a minimizer of $\mathcal D^1_\mu$. 
	
	We now proceed to show that $u_\infty$ solves \eqref{eq: KE_scal_eq}. Indeed, for any $f\in C^0(X,\RR)$, we have
	$$
	\mathcal D^\lambda_\mu(u_\infty)\leq \mathcal D^\lambda(P_\theta(u_\infty+tf))\leq\frac{-1}{\lambda}\log \int_X e^{-\lambda(u_\infty+tf)}\mathrm d\mu-I_\theta(P_\theta(u_\infty+tf))
	$$
	for all $t\in\RR$. So the expression
	$
	\frac{-1}{\lambda}\log \int_X e^{-\lambda(u_\infty+tf)}\mathrm d\mu-I_\theta(P_\theta(u_\infty+tf))
	$
is minimized at $t=0$. By \cite[Lemma 4.2]{BBGZ13} we can take derivaties and conclude that
$$
\frac{\int_X f e^{-\lambda u_\infty}\mathrm d\mu}{\int_Xe^{-\lambda u_\infty}\mathrm d\mu}=\frac{\int_X f\theta^n_{u_\infty}}{\vol(\{\theta\})}
$$
for any $f\in C^0(X,\RR)$. This implies that
$
\frac{e^{-\lambda u_\infty}\mathrm d\mu}{\int_Xe^{-\lambda u_\infty}\mathrm d\mu}=\frac{\theta^n_{u_\infty}}{\vol(\{\theta\})},
$
finishing the argument.
\end{proof}

Now we relate the properness of $\lambda$-twisted Ding functional to destabilizing geodesic rays. For $\lambda = 1$ a similar result appears in \cite[Theorem 2.16]{BBJ21} in the ample case. One of the novelties of our work is that the argument below does not use the (convexity of the) K-energy, allowing us to extend the scope of the result significantly.

\begin{theorem}
\label{thm:non-proper-imply-destablizing-ray} Let $\lambda\in(0,c_\mu[V_\theta])$. Assume that $\mathcal{D}^\lambda_\mu$ is not proper, $\mathcal L_\mu^\lambda$ is $L^1$-continuous on $\mathcal E^1(X,\theta)$, and  $\mathcal{D}^\lambda_\mu$ is convex along the geodesics of $\mathcal{E}^1(X,\theta)$. Then there exists a non-trivial finite energy geodesic ray $\{v_t\}_t\in\mathcal{R}^1(X,\theta)$ such that $\mathcal{D}_\mu^\lambda\{v_t\}\leq0$.
\end{theorem}

\begin{proof}
    
Since $\mathcal L_\mu^\lambda (u) = \mathcal D^\lambda_\mu(u) + I_\theta(u)$, we notice that $\mathcal L_\mu^\lambda$ is also convex along geodesics, as $I_\theta$ is linear along them. Recall also that $I_\theta$ is $L^1$-usc \cite{BBEGZ16}, so 
$\mathcal{D}^\lambda_\mu$ is $L^1$-lsc.

Since $\mathcal D^\lambda_\mu$ is not proper, there exists $\varphi^j \in \mathcal E^1(X,\theta)$ with $\sup_X\varphi^j=0$ satisfying
$$
\mathcal{D}^\lambda_\mu(\varphi^j)\leq\frac{1}{j}d_1(V_\theta,\varphi^j)-j \text{ for }j\in\mathbb{N}_{>0},
$$
where $d_1(V_\theta,\varphi^j)=I_\theta(V_\theta)-I_\theta(\varphi^j)=-I_\theta(\varphi^j)$.
We claim that  $d_1(V_\theta,\varphi^j)\rightarrow\infty$. Indeed, suppose that $I_\theta(\varphi^j)\geq -C$ for all $j$. Since $I_\theta$ is $L^1$-usc, after possibly passing to a subsequence, one can find $\varphi^\infty\in\mathcal{E}^1(X,\theta)$ with $\varphi^j\to\varphi^\infty$ in the $L^1$ topology. Using that $\mathcal{D}^\lambda_\mu$ is $L^1$-lsc, we deduce $-\infty<\mathcal{D}^\lambda_\mu(\varphi^\infty)\leq\liminf_j\mathcal{D}^1_\mu(\varphi^j)=-\infty$, a contradiction.

As a result, using $d_1(V_\theta,\varphi^j) = -I_\theta(\varphi^j)$, we get that  
$$\limsup_j \frac{\mathcal L_\mu^\lambda(\varphi^j)}{d_1(V_\theta,\varphi^j)} =\limsup_j \frac{\mathcal D_\mu^\lambda(\varphi^j) - d_1(V_\theta,\varphi^j)}{d_1(V_\theta,\varphi^j)} \leq -1 \ \ \textup{ and } \ \ d_1(V_\theta, \varphi^j) \to \infty.$$

Let $[0,d_1(V_\theta,\varphi^j)] \ni t \mapsto u^j_t \in \mathcal E^1(X,V_\theta)$ be the unit speed finite energy geodesic joining $V_\theta, \varphi^j$. Due to the convexity of $\mathcal L_\mu^\lambda$ along such geodesic segments we have that
$$\limsup_j \frac{\mathcal L_\mu^\lambda(u^j_t)}{t} \leq -1, \ \ d_1(V_\theta,u^j_t) = -I_\theta(u^j_t) = t  \ \ \textup{ for all }t \in [0,d_1(V_\theta,\varphi^j)].$$

As the geodesic segments $t \to u^j_t$ emenate from $V_\theta$ we have $\sup_X(u^j_t - V_\theta) = \sup_X u^j_t=0, t \in [0,d_1(V_\theta,\varphi^j)]$, as follows from  Remark \ref{rem: sup_ray}.
It is well known that the condition $\sup_X u^j_t=0$ implies that the $L^1$ norm on $X$ of $u^j_t$ is uniformly bounded  with respect to $j,t$ (see for example \cite[Lemma 3.45]{Dar19}). By Fubini's theorem, the $L^1$ norm of $(t,x) \to u^j_t(x)$ is uniformly bounded on any compact subset $K \subset (0,\infty) \times X$  (for high enough $j$ we have $K \subset (0,d_1(V_\theta,\varphi^j)) \times X)$).

Let $S_j := (0,d_1(V_\theta,\varphi^j)) + i \mathbb R, \ S := (0,\infty) + i \mathbb R \subset \mathbb C$. We get that 
$$U^j(s,x) := u^j_{\textup{Re }s}(x) \in \textup{PSH}(S_j \times X, \pi^* \theta).$$

Due to $i \mathbb R$-invariance, on any compact subset of $S \times X$ the $L^1$-norms of $U^j$ are uniformly bounded. As a result, we can apply \cite[Proposition I.5.9]{Dem12} to conclude existence of $U \in \textup{PSH}(S \times X, \pi^* \theta)$ such that $U^j \to U$ with respect to $L^1_{\textup{loc}}(S \times X)$. Since $U$ is $i\mathbb R$-invariant we obtain existence of a subgeodesic ray $\{u_t\}_t$ such that $U(s,x) = u_{\textup{Re }s}(x)$.

Next we argue (among other things) that $\{u_t\}$ is actually a finite energy sublinear subgeodesic. By $i\mathbb R$-invariance we have that $u^j_t(x) \to u_t(x)$ with respect to $L^1_\textup{loc}((0,\infty) \times X)$, hence by Fubini's theorem we obtain that $u^j_t$ converges to $u_t$ with respect to $L^1(X)$, for all $t \in (0,\infty) \setminus E$, where $E$ is a set of Lebesgue measure zero.
Thus
\begin{equation}\label{eq: limit_ray_estimates}
\frac{\mathcal L_\mu^\lambda (u_t)}{t} \leq - 1, \ \ \sup_X u_t =0, \ \ 0 \geq I_\theta(u_t) \geq -t, \textup{ for all } t \in (0,\infty) \setminus E,
\end{equation}
as $\mathcal L_\mu^\lambda$ is $L^1$-continuous, $\sup_X(\cdot)$ is $L^1$-continuous and $I_\theta$ is $L^1$-usc \cite{BBEGZ16}.

Let $l_j \in (0,\infty) \setminus E$ be a sequence converging to some $t \in (0,\infty)$. 
Due to $t$-convexity of $\{u_t\}_t$, and the fact that $\sup_X u_{l_j} =0$, we must have that $u_{l_j}$ $L^1$-converges to $u_t$. As a result, we have that \eqref{eq: limit_ray_estimates} must hold for all $t \in (0,\infty)$.

In addition,  due to \cite[Theorem 3.11]{DDNL3}, we see that $u_t\to V_\theta$ in $L^1$ as $t \to 0$, as we have $d_1(u_t,V_\theta) = - I_\theta(u_t) \to 0$ as $t \to 0$. By the definition of $\mathcal L_\mu^\lambda$ and \eqref{eq:def-L} we get that 
$$
\mathcal L_\mu^\lambda\{u_t\} = I_\theta\{u_t\} + {\mathcal D}^\lambda_\mu\{u_t\}\leq -1.
$$ 

Let $\{v_t\}_t$ be the maximization of $\{u_t\}_t$ defined in \eqref{eq: maximization_def}. By \eqref{eq:radial-I-formula-subgeo-rays} and \eqref{eq: limit_ray_estimates}, we have that
$I_\theta\{v_t\}=I_\theta\{u_t\}\geq-1.
$
By Theorem \ref{thm: radial_ding} and Proposition \ref{prop:c-u=c-P-u}, we then have that
$
{\mathcal D}^\lambda_\mu\{v_t\}={\mathcal D}^\lambda_\mu\{u_t\}\leq 0$ and also ${\mathcal L}^\lambda_\mu\{v_t\}={\mathcal L}^\lambda_\mu\{u_t\} \leq -1.
$

To finish we argue that $\{v_t\}_t$ is non-trivial. Since $\sup_X v_t = \sup_X u_t =0$, we only need to rule out that $v_t= V_\theta, \ t \geq 0$. If this were the case, then ${\mathcal L}^\lambda_\mu\{v_t\}=0$  for all $t\geq0$. This would imply that $0={\mathcal L}^\lambda_\mu\{v_t\} = \mathcal L_\mu^\lambda\{u_t\}  \leq -1$, which is absurd. 
\end{proof}

Lastly, using a `thermodynamic argument' (cf. \cite[Propsition 4.11]{BBEGZ16} and \cite[Proposition 3.5]{Zhang21}), one can  show that properness of the $\lambda-$Ding functional is an open condition. This result will only be used to prove Proposition \ref{prop:unique-imply-delta>1}.

\begin{proposition}
\label{prop:proper-open}
    If $\mathcal{D}^\lambda_\mu$ is proper, then so is $\mathcal{D}^{\lambda + \tau}_\mu$ for $\tau>0$ small. In particular,   $\delta_\mu>\lambda$.
\end{proposition}

\begin{proof}

Since $\mathcal{D}^\lambda_\mu$ is coercive, there exist $\varepsilon>0$ and $C_0>0$ such that
$$
-\frac{1}{\lambda}\log\int_Xe^{-\lambda v}\mathrm{d}\mu-I_\theta(v)\geq\varepsilon\big(\sup_X v-I_\theta(v)\big)-C_0\text{ for all }v\in\mathcal{E}^1(X,\theta).
$$
Using that $\sup_X v=\sup_X(v-V_\theta)$, we derive
$$
\sup_X v-I_\theta(v)\geq\frac{1}{\vol(\{\theta\})}\int_X(v-V_\theta)\theta_{V_\theta}^n-I_\theta(v)\geq\frac{1}{n}\bigg(I_\theta(v)-\frac{1}{\vol(\{\theta\})}\int_X(v-V_\theta)\theta_v^n\bigg),
$$
where the last inequality follows from \cite[(2.7)]{BBGZ13}. Thus we have, for any $v\in\mathcal{E}^1(X,\theta)$,
\begin{equation}
\label{eq:Ding-proper-esti}
    -\frac{1}{\lambda}\log\int_Xe^{-\lambda v}\mathrm{d}\mu-E_\theta(v)\geq\frac{\varepsilon}{n}\bigg(E_\theta(v)-\frac{1}{\vol(\{\theta\})}\int_X(v-V_\theta)\theta_v^n\bigg)-C_0.
\end{equation}
Recall that $\lambda<c_\mu[V_\theta]$,
we may fix $\varepsilon$ small enough so that $\lambda(1+2\varepsilon/n)<c_\mu[V_\theta]$ and put
\begin{equation*}
    \label{eq: c1_est}
    C_1:=\int_Xe^{-\lambda(1+\varepsilon/n)V_\theta}\mathrm{d}\mu<\infty.
\end{equation*}

Our goal is to find a constant $\tau>0$ such that
\begin{equation}
    \label{eq:MT-ineq}
    \sup_{\varphi\in\mathcal{E}^1(X,\theta)}\int_Xe^{-\lambda(1+\tau)(\varphi-I_\theta(\varphi))}\mathrm{d}\mu<\infty.
\end{equation}
Note that this is equivalent to $\inf_{\varphi\in\mathcal{E}^1(X,\theta)}D^{\lambda(1+\tau)}_\mu(\varphi)>-\infty$, which implies that $\delta_\mu\geq\lambda(1+\tau)$ by Theorem \ref{thm:delta=sup-D>0}. Moreover, in view of \eqref{eq: D_twist_ineq}, $\mathcal{D}_\mu^{\lambda^\prime}$ is proper for any $\lambda^\prime\in(0,\lambda(1+\tau))$, hence finishing the proof.

So the rest of the argument is devoted to showing \eqref{eq:MT-ineq}. Given any $\varphi\in\mathcal{E}^1(X,\theta)$, \cite[Theorem B]{BEGZ10} ensures that one can find $u\in\mathcal{E}^1(X,\theta)$ solving the equation
\begin{equation}
\label{eq:u-solve-MA}
    \theta_{u}^n=e^{c-\lambda(1+\varepsilon/n)\varphi+\frac{\varepsilon\lambda}{n}V_\theta}\mathrm{d}\mu
\end{equation}
for some suitable normalization constant $c\in\RR$, since the above right hand side has $L^p$ density for some $p>1$ (recall $\lambda(1+\varepsilon/n)<c_\mu[V_\theta]=c_\mu[\varphi]$ and openness \cite{Bern15o,GuZh15}). 

We put
$$
H_\mu(u):=\frac{1}{\vol(\{\theta\})}\int_X\log\bigg(\frac{\theta_u^n}{\vol(\{\theta\})\mathrm{d}\mu}\bigg)\theta_{u}^n.
$$
One can then easily check the following:
\begin{flalign*}
H_\mu(u) &\leq\frac{1}{\vol(\{\theta\})}\int_X (-\lambda(1+\varepsilon/n)\varphi)e^{c-\lambda(1+\frac{\varepsilon\lambda}{n})\varphi}\mathrm{d}\mu+A\leq A \int_X e^{-\lambda(1+\frac{2\varepsilon}{n})\varphi}\omega^n + A< \infty
\end{flalign*}
for some constant $A>0$,
where the first inequality is due to $V_\theta\leq0$, in the penultimate estimate we have used that 
$t e^{-\frac{\varepsilon}{n} t}$ is always bounded for all $t \geq 0$, and in the last estimate we have used $\lambda(1+2\varepsilon/n)<c_\mu[V_\theta]$ and Proposition \ref{prop:c-u=c-P-u}. So $H_\mu(u)$ does make sense. Moreover, Jensen's inequality implies that, for any $\alpha\in(0,\alpha_\mu)$ (recall \eqref{eq:def-alpha}) and $\phi\in\textup{PSH}(X,\theta)$ with $\sup_X\phi=0$,
$$
\int_X(-\alpha \phi)\theta^n_u\leq\log\int_Xe^{-\alpha \phi}\mathrm{d}\mu+H_\mu(u)<\infty.
$$
In other words, any $\phi\in\textup{PSH}(X,\theta)$ is $L^1$-integrable against the measure $\theta^n_{u}$.

Applying Jensen's inequality once again, we deduce from
\eqref{eq:Ding-proper-esti} that
$$
\frac{1}{\vol(\{\theta\})}\int_X u\theta_u^n+\frac{1}{\lambda}H_\mu(u)-I_\theta(u)\geq\frac{\varepsilon}{n}\bigg(I_\theta(u)-\frac{1}{\vol(\{\theta\})}\int_X(u-V_\theta)\theta_u^n\bigg)-C_0.
$$
Namely,
$$
H_\mu(u)\geq\lambda(1+\frac{\varepsilon}{n})\bigg(I_\theta(u)-\frac{1}{\vol(\{\theta\})}\int_X(u-V_\theta)\theta^n_{u}\bigg)-C_0\lambda-\frac{\lambda}{\vol(\{\theta\})}\int_XV_\theta\theta^n_{u}.
$$
Combining \eqref{eq:u-solve-MA} with this estimate, we can write
\begin{equation*}
    \begin{aligned}
    \log\int_X& e^{-\lambda(1+\frac{\varepsilon}{n})(\varphi-I_\theta(\varphi))+\frac{\varepsilon\lambda}{n}V_\theta}\mathrm{d}\mu=\log\int_X\frac{e^{-\lambda(1+\frac{\varepsilon}{n})\varphi+\frac{\varepsilon\lambda}{n}V_\theta}\mathrm{d}\mu}{\theta^n_{u}/\vol(\{\theta\})}\frac{\theta^n_{u}}{\vol(\{\theta\})}+\lambda(1+\frac{\varepsilon}{n})I_\theta(\varphi)\\
    &=\int_X\log\bigg(\frac{e^{-\lambda(1+\frac{\varepsilon}{n})\varphi+\frac{\varepsilon\lambda}{n}V_\theta}\mathrm{d}\mu}{\theta^n_{u}/\vol(\{\theta\})}\bigg)\frac{\theta^n_{u}}{\vol(\{\theta\})}+\lambda(1+\frac{\varepsilon}{n})I_\theta(\varphi)\\
    &=\int_X\big(-\lambda(1+\frac{\varepsilon}{n})\varphi\big)\frac{\theta^n_{u}}{\vol(\{\theta\})}-H_\mu(u)+\frac{\varepsilon\lambda}{n}\int_XV_\theta\frac{\theta^n_{u}}{\vol(\{\theta\})}+\lambda(1+\frac{\varepsilon}{n})I_\theta(\varphi)\\
    &\leq\lambda(1+\frac{\varepsilon}{n})\bigg(I_\theta(\varphi)-I_\theta(u)-\int_X(\varphi-u)\frac{\theta^n_{u}}{\vol(\{\theta\})}\bigg)+C_0\lambda.\\
    \end{aligned}
\end{equation*}
Now using that 
$
I_\theta(\varphi)-I_\theta(u)\leq\frac{1}{\vol(\{\theta\})}\int_X(\varphi-u)\theta_u^n\text{ for }\varphi,u\in\mathcal{E}^1(X,\theta),
$ (see e.g. \cite[Theorem 2.4(3)]{DDNL3})
we arrive at
$$
\int_Xe^{-\lambda(1+\frac{\varepsilon}{n})(\varphi-I_\theta(\varphi))+\frac{\varepsilon\lambda}{n}V_\theta}\mathrm{d}\mu\leq e^{C_0\lambda} \text{ for any }\varphi\in\mathcal{E}^1(X,\theta). 
$$

Now we are able to prove \eqref{eq:MT-ineq}. Take $\tau:=\frac{\varepsilon^2}{n^2+2\varepsilon n}$, then
$
1+\tau=\frac{(1+\varepsilon/n)^2}{1+2\varepsilon/n}<1+\frac{\varepsilon}{n}.
$
So H\"older's inequality with $p=\frac{1+2\varepsilon/n}{1+\varepsilon/n}$ and $q  = \frac{1+2\varepsilon/n}{\varepsilon/n}$ implies that
\begin{equation*}
    \begin{aligned}
    \int_X & e^{-\lambda(1+\tau)(\varphi-I_\theta(\varphi))}\mathrm{d}\mu=\int_Xe^{-\lambda(1+\tau)(\varphi-I_\theta(\varphi))+\frac{\lambda\varepsilon(1+\tau)}{n+\varepsilon
    }V_\theta-\frac{\lambda\varepsilon(1+\tau)}{n+\varepsilon
    }V_\theta}\mathrm{d}\mu\\
    &\leq\bigg(\int_Xe^{-\lambda(1+\frac{\varepsilon}{n})(\varphi-I_\theta(\varphi))+\frac{\varepsilon\lambda}{n}V_\theta}\mathrm{d}\mu\bigg)^{\frac{1+\varepsilon/n}{1+2\varepsilon/n}}\bigg(\int_Xe^{-\lambda(1+\frac{\varepsilon}{n})V_\theta}\mathrm{d}\mu\bigg)^{\frac{\varepsilon/n}{1+2\varepsilon/n}}\\
    &\leq\big(e^{C_0\lambda}\big)^{\frac{1+\varepsilon/n}{1+2\varepsilon/n}}\cdot C_1^{\frac{\varepsilon/n}{1+2\varepsilon/n}}.
    \end{aligned}
\end{equation*}
This completes the proof.
\end{proof}

\subsection{The proof of Theorem \ref{thm: main_KE_twist_exist}}
In this part we return to the assumptions at the very begining of this paper, that  $c_1(-K_X)$ admits a decomposition
$
c_1(-K_X)=\{\theta\}+\{\eta\},
$
where $\eta$ is a smooth representative of a pseudoeffective class and $\psi \in \textup{PSH}(X,\eta)$. In particular, $-K_X$ is big (thus $X$ is projective). Further assume that  $\eta + \ddc \psi \in\{\eta\}$ has klt singularities, i.e., $\int_Xe^{-\psi}\omega^n<\infty.$ Our goal is to search for a twisted KE potential $u\in\mathcal{E}^1(X,\theta)$ satisfying
$$
\Ric \theta_{u}=\theta_{u}+\eta_\psi.
$$
This amounts to solving
\begin{equation}\label{eq: KE_scal_eq}
\theta_u^n=e^{f-u-\psi}\omega^n,
\end{equation}
with $f \in C^\infty(X)$ a Ricci potential satisfying $\Ric \omega = \theta + \eta + \ddc f.$
So in what follows we choose our tame measure $\mu$ to be
$$
\mu:=e^{f-\psi}\omega^n.
$$

\begin{proposition}
\label{thm:geodesic_stability_implies_properness} Suppose that $c_\mu[V_\theta]>1$ and that $\mathcal D^1_\mu$ is not proper. Then $\delta_\mu\leq 1$. 
\end{proposition}

\begin{proof} 
Since $-K_X$ is big under our conditions, we have that $X$ is Moishezon \cite{JiSh93}, hence also projective \cite{Mo66}. Also, the condition $c_\mu[V_\theta] > 1$ implies that 
$c_\mu[v] > 1$ for any $v \in \mathcal E^1(X,\theta)$ (Proposition \ref{prop:c-u=c-P-u}). As a result, given a finite energy geodesic  $t \mapsto v_t$, 
both conditions of  \cite[Theorem 0.1]{BePa08} are satisfied, and we obtain that $t \mapsto \mathcal D^1_\mu(v_t)$ is convex. 

By openness \cite{Bern15o,GuZh15} and \cite[Theorem 0.2(2)]{DK01} we get that $\mathcal L_\mu^1(u)$ (recall \eqref{eq: D_L_def}) is $L^1$-continuous. Now, applying Theorem \ref{thm:non-proper-imply-destablizing-ray}, we find a unit speed finite energy geodesic ray $\{w_t\}_t$ with $\sup_X w_t=0$ such that $\mathcal D^1_\mu\{w_t\}\leq 0$.
Then by Proposition \ref{prop:delta>1=>geodesic-stable} one must have $\delta_\mu\leq 1$, as asserted.
\end{proof}

A direct consequence is the following:

\begin{corollary}[=Theorem \ref{thm: main_KE_twist_exist}]\label{cor: main_KE}
If $\delta_\mu > 1$ then $\mathcal D^1_\mu$ is proper. In particular,  \eqref{eq: KE_scal_eq} admits a solution $u \in \textup{PSH}(X,\theta)$ with minimal singularity type.
\end{corollary}
\begin{proof} Lemma \ref{lem: c_delta_est} gives that $c_\mu[V_\theta]>1$. By the previous result and its proof, $\mathcal D_\mu^1$ is proper and $\mathcal{L}^1_\mu$ is $L^1$-continuous. Thus, by Proposition \ref{prop: proper_implies_minimizer}, \eqref{eq: KE_scal_eq} admits a solution $u \in \mathcal E^1(X,\theta)$. 
That $u$ has minimal singularity type follows from Kolodzej's estimate (see \cite[Theorem B]{BEGZ10}), as the right hand side of \eqref{eq: KE_scal_eq} has $L^p$ density for some $p>1$ by openness \cite{Bern15o,GuZh15}.
\end{proof}

\paragraph{The (partial) converse. }

In this paragraph we give a (partial) converse for Corollary \ref{cor: main_KE}, and show the role of uniqueness in a full converse. We begin with the next standard result, following \cite[Theorem 6.6]{BBGZ13}.

\begin{proposition} If $u \in \mathcal E^1(X,\theta)$ is a solution to \eqref{eq: KE_scal_eq} then $u$ minimizes $\mathcal D^1_\mu$.
\end{proposition}

\begin{proof}
Given $u,v\in\mathcal{E}^1(X,\theta)$, let $[0,1]\ni t\mapsto u_t\in\mathcal E^1(X,\theta)$ be the geodesic segment connecting $u_0 = u$ and $u_1 =v$, constructed in \cite{DDNL3}. 
If $u$ satisfies 
$
\theta^n_u=e^{-u}\mu
$
then  $c_\mu[V_\theta]>1$ by openness \cite{Bern15o,GuZh15}, so $\mathcal{D}^1_\mu$ is well defined on $\mathcal{E}^1(X,\theta)$ by Proposition \ref{prop:c-u=c-P-u}.
We introduce
$$
f(t):=\mathcal D^1_\mu(u_t),\ t\in[0,1].
$$
By $d_1$-continuity of the Ding energy and \cite[Theorem 0.1]{BePa08}, $f$ is continuous and convex.

We want to show that $\mathcal D^1_\mu(u)\leq \mathcal D^1_\mu(v)$. Note that $u$ has minimal singularity by \cite{BEGZ10}. We can assume that $v$ also has minimal singularity type, since such potentials are $d_1$-dense in $\mathcal E^1(X,\theta)$. By the above considerations, it is enough to argue that
$$
\frac{d}{dt}\bigg|_{t=0^+} f(t)\geq0.
$$

After adding a constant to $v$, we may further assume that $u>v$ on $X$. 
Consider
$
g(t):=\frac{u_t-u}{t}\text{ for }t\in(0,1].
$
Since $u,v$ have minimal singularity type,  $t \to u_t$ is uniformly $t$-Lipschitz. As a result, $\lim_{t \to 0} g(t) =: \dot u_0^+$ is uniformly bounded and well defined away from a pluripolar set. We notice that (by \cite[Theorem 2.4(iii)]{DDNL3})
$$
\frac{I_\theta(u_t)-I_\theta(u)}{t}\leq\frac{1}{\vol(\{\theta\})}\int_{X} g(t)\theta^n_{u}=\frac{1}{\vol(\{\theta\})}\int_X g(t) e^{-u}\mathrm d\mu.
$$
Letting $t\rightarrow 0$ by monotone convergence we get
\begin{equation}\label{eq: I_deriv}
\frac{d}{dt}\bigg|_{t=0^+} I_\theta(u_t)\leq\frac{1}{\vol(\{\theta\})}\int_{X} \dot u_0^+ e^{-u}\mathrm d\mu.
\end{equation}
On the other hand, using dominated convergence we conclude that
\begin{equation}\label{eq: L_deriv}
\lim_{t \searrow 0}\frac{\int_Xe^{-u_t}\mathrm d\mu-\int_Xe^{-u}\mathrm d\mu}{t}= \lim_{t \searrow 0 }\int_X\frac{u-u_t}{t}\frac{e^{u-u_t}-1}{u-u_t}e^{-u}\mathrm d\mu=\int_X(-\dot u_0^+)e^{-u}\mathrm d\mu.
\end{equation}
Putting together \eqref{eq: I_deriv} and \eqref{eq: L_deriv}, we conclude that  $
\frac{d}{dt} \big| _{t=0^+} \mathcal D^1_\mu(u_t)\geq0,
$ finishing the proof.
\end{proof}
We now record a partial converse to Corollary \ref{cor: main_KE}:

\begin{proposition}\label{prop: partial_conv}
Assume that \eqref{eq: KE_scal_eq} is solvable. Then $\delta_\mu\geq 1$.
\end{proposition}

\begin{proof}
As follows from the previous result (and its argument), $\mathcal{D}^1_\mu$ is well defined and bounded from below on $\mathcal{E}^1(X,\theta)$. This gives $\delta_\mu\geq 1$ in light of Theorem \ref{thm:delta=sup-D>0}.
\end{proof}

We end this section with the following full converse to Corollary \ref{cor: main_KE}, provided the uniqueness of the solution.

\begin{proposition}
\label{prop:unique-imply-delta>1}
    If the twisted KE metric found in Corollary \ref{cor: main_KE} is unique, then $\delta_\mu>1$
\end{proposition}

\begin{proof}
    Indeed, if a unique twisted KE metric exists then the conditions of the properness/existence principle of the first named author and Y. Rubinstein  \cite[Theorem 3.7]{DR17} (c.f. \cite[Theorem 4.7]{Dar19}) are satisfied, implying that the Ding functional $\mathcal{D}^1_\mu$ is proper.  Then Proposition \ref{prop:proper-open} ensures that $\delta_\mu>1$, as asserted.
\end{proof}

\section{The case of Fano type varieties}
\label{sec:log-Fano}

In this section we show  how our techniques yield a simplified proof of the YTD existence theorem of Li--Tian--Wang. Conversely, we also elaborate on the recent note of C. Xu, pointing out how the Li--Tian--Wang result together with deep results of the minimal model gives an alternative argument of our Corollary 1.3.

A log Fano pair $(Z,\Delta)$ consists of the following data:
$Z$ is a normal projective variety and $\Delta$ an effective Weil divisor on $Z$ such that  $-K_Z-\Delta$ is an ample $\mathbb Q$-Cartier divisor and that $(Z,\Delta)$ has klt singularities.

\subsection{The YTD existence theorem of Li--Tian--Wang.}
In what follows we give a simplified proof of the following result due to Li--Tian--Wang \cite{LTW21b}, using the techniques of our paper:

\begin{theorem}
\label{thm:LTW}
The log Fano pair $(Z,\Delta)$ admits a KE metric if it is uniformly K-stable.
\end{theorem}

Put for simplicity $L:=-K_Z-\Delta$.
Then the pair $(Z,\Delta)$ being uniformly K-stable means exactly that $\delta_\Delta(L)>1$ (cf. \cite{BJ17}), where
$$
\delta_\Delta(L):=\inf_E\frac{A_{Z,\Delta}(E)}{S_L(E)}.
$$
Here the infimum is over all prime divisors $E\subset Y\xrightarrow{\pi}Z$ over $Z$, where $Y$ is a smooth projective manifold and $\pi$ is a projective birational morphism.  The log discrepancy is 
$
A_{Z,\Delta}(E):=1+\ord_E(K_Y+\pi^*L)
$
and the expected vanishing order $S_L(E)$ is given by
$$
S_L(E):=\frac{1}{\vol(L)}\int_0^{\tau_L(E)}\vol(\pi^*L-xE)\mathrm{d}x,
$$
with $\tau_L(E):=\sup\{x>0:\pi^*L-xE\text{ big}\}$ being the pseudoeffective threshold. The volume function $\vol(\cdot)$ here is understood in the sense of algebraic geometry (cf. \cite{Laz-bookI}).

To prove that $(Z,\Delta)$ admits a KE metric, we use the variational framework in \cite{BBJ21,LTW21b}. The goal is to argue that the corresponding Ding functional $\textbf{D}$ defined in \cite[(40)]{LTW21b} for the pair $(Z,\Delta)$ is proper. Suppose on the contrary that $\textbf{D}$ is not proper. To derive a contradiction out of this, the Li--Tian--Wang argument takes five steps; see \cite[\S 4.1-4.5]{LTW21b}, where the first step showing the convexity of the Mabuchi functional is particularly difficult to establish. 

However, relying on the results in this paper we can obtain a contradiction almost immediately. Indeed, take a log resolution $X\xrightarrow{\sigma}Z$ of the pair $(Z,\Delta)$, so in particular $X$ is a smooth birational model over $Z$.
 We take our big class $\{\theta\}$ to be $c_1(\sigma^*L)$ and let $\mu=e^{\chi-\psi}\omega^n$ be the tame measure on $X$ that is associated to the log Fano pair $(Z,\Delta)$ (see \cite[Lemma 3.2(i)]{BBEGZ16}, with $\chi:=\psi^+$,  $\psi:=\psi^-$ and $\omega^n:=dV$ for some K\"ahler form $\omega$ on $X$). Then we consider the following algebraic $\delta$-invariant:
 $$
 \Tilde{\delta}_\mu(\sigma^*L):=\inf_E\frac{A_{\chi,\psi}(E)}{S_{\sigma^*L}(E)},
 $$
 where the inf is over all prime divisors in smooth birational models over $X$. Note that this invariant is potentially different from the $\delta_\mu$ invariant defined in \eqref{eq: delta_mu_def}, as the latter involves inf over all smooth bimeromorphic models over $X$. However due to \cite[Theorem B.7]{BBJ21} and  the  non-pluripolar volume being the same as the  algebraic volume in the Ner\'on--Severi space (see \cite{Bo02} and \cite[Proposition 1.18]{BEGZ10}), we find that all the previous arguments in this paper hold without change when restricting to divisors in smooth birational models over $X$. Hence, Theorem \ref{thm:delta=sup-D>0}  (or simply Theorem \ref{thm:delta=inf-c[psi]'}) gives that
 $$
\delta_\mu(\{\theta\})=\Tilde{\delta}_\mu(\sigma^*L).
 $$
 Moreover, as any prime divisor $E$ over $X$ is a prime divisor over $Z$, one has $A_{\chi,\psi}(E)=A_{Z,\Delta}(E)$ and $S_{\sigma^*L}(E)=S_L(E)$. So we find that
$$
\delta_\mu(\{\theta\})\geq\delta_\Delta(L).
$$
(They are actually equal, as by Hironaka's theorem any prime divisor over $Z$ can be viewed as a prime divisor over $X$ as well.)

Next, consider the Ding functional $\mathcal{D}^1_\mu$ associated to the above chosen triple $(X,\{\theta\},\mu)$. It follows from Berndtsson's convexity that $\mathcal{D}^1_\mu$ is convex along geodesics in $\mathcal{E}^1(X,\theta)$ (see \cite[Theorem 11.1]{BBEGZ16} for the precise version that we need). Note also that $V_\theta=0$ (as $\{\theta\}$ is semipositive) and $\int_Xe^{-p\psi}\omega^n<\infty$ for some $p>1$ (by \cite[Lemma 3.2(i)]{BBEGZ16}). So by Proposition \ref{prop:c-u=c-P-u} one has $\int_Xe^{-pu-p\psi}\omega^n<\infty$ for any $u\in\mathcal{E}^1(X,\theta)$.
As a result, \cite[Theorem 0.2(2)]{DK01} implies that $u\mapsto\int_Xe^{-u-\psi}\omega^n$ is $L^1$-continuous on $\mathcal{E}^1(X,\theta)$. This argument also implies that $\mathcal L^1_\mu$ is $L^1$-continuous (recall \eqref{eq: D_L_def}), as $\chi$ is bounded from above.

Moreover, $\mathcal{D}^1_\mu$ is not proper since we assumed the Ding functional $\mathbf{D}$ on $Z$ is not proper. 
Then by Theorem \ref{thm:non-proper-imply-destablizing-ray}, we obtain a unit speed finite energy geodesic ray $\{w_t\}_t$ with $\sup_X w_t=0$ such that
$
\mathcal{D}^1_\mu\{w_t\}\leq0.
$

On the other hand, $(Z,\Delta)$ being uniformly K-stable implies that $\delta_\mu(\{\theta\})\geq\delta_\Delta(L)>1$. So we derive from Proposition \ref{prop:delta>1=>geodesic-stable} that
$
\mathcal{D}^1_\mu\{w_t\}>0.
$
This immediately gives us a contradiction, hence proving Theorem \ref{thm:LTW}.

\subsection{Connections with birational geometry}
\label{sec:ample-model}

In this part we restrict ourselves to the case when $\{\theta\}=c_1(-K_X)$. So $X$ is projective and $-K_X$ is big. In this particular setting, following \cite{Xu22} closely, we point out how our work connects with birational geometry to give an algebraic proof of Corollary 1.3. 

%When $-K_X$ is big and nef, $X$ is a crepant resolution of a Fano variety $Z$ with only canonical singularities, by Kawamata's base point freeness. In this case the KE metrics on $X$ are simply the pullbacks of KE metrics on $Z$ (we are grateful to Y. Odaka for pointing this out). 

When $-K_X$ is big, the anti-canonical ring $
R(X,-K_X):=\bigoplus_{m\geq 0}H^0(X,-mK_X)
$ might not be finitely generated. Surprisingly, C. Xu recently observed that when $\delta(-K_X)>1$, the ring $R(X,-K_X)$ is indeed finitely generated:

\begin{lemma}\cite[Lemma 3.1]{Xu22} \label{lem: Xu_lemma}
    Assume that $-K_X$ is big and $\delta(X,-K_X)>1$. Then there exists an effective $\mathbb{Q}$-divisor $D\sim_\mathbb{Q}-K_X$ such that the pair $(X,D)$ is klt. In particular $X$ is a Mori dream space, hence $
R(X,-K_X)$ is finitely generated.
\end{lemma}

The last statement follows from \cite[Corollary 1.3.1]{BCHM10}. In what follows we point out how this result gives an alternative proof for Corollary 1.3. What is more, using Proposition \ref{prop:unique-imply-delta>1}, in this case we get a full characterization:

\begin{theorem}
\label{thm:unique-KE}
Suppose that $\{\theta\}=c_1(-K_X)$ is big. Then $\Ric \theta_u=\theta_u$ has a unique solution $u \in \textup{PSH}(X,\theta)$ with minimal singularity if and only if $\delta(-K_X) > 1$.    
\end{theorem}

By Proposition \ref{prop:unique-imply-delta>1} we only need to argue existence and uniqueness. So we assume that $\delta(-K_X) > 1$ and let $r>0$ be sufficiently divisible such that
$\bigoplus_{m\geq 0}H^0(X,-mrK_X)$ is generated by $H^0(X,-rK_X)$. Standard results in birational geometry (see e.g., \cite[Proposition 1.16]{Kol13}) imply a number of facts that we now recall. First, $|-rK_X|$ induces a birational map, say $\phi:X\dasharrow Z$, from $X$ to a normal projective variety $Z$. Next, letting $H$ denote the hyperplane class on $Z$, then $-rK_Z\sim H$, so in particular $-K_Z$ is an ample $\mathbb{Q}$-line bundle. Moreover, letting $X\xleftarrow{\pi} W\xrightarrow{\tau} Z$ be a resolution of the base locus of $|-rK_X|$,  one has a following Zariski decomposition     $\pi^*|-mrK_X|=\tau^*|mH|+mF\text{ for all }m\geq0,$
where $F\geq0$ is a $\tau$-exceptional divisor on $W$. So in particular, one has a Zariski decomposition for the big line bundle $-K_X$:
\begin{equation}
    \label{eq:Zariski-decom-line-bundle}
    -\pi^*K_X=-\tau^*K_Z+\frac{1}{r}F.
\end{equation}
Finally, one can show that $Z$ has klt singularities. Indeed, it is enough to argue that $A_Z(E)>0$ for any prime divisor $E$ over $Z$. For this, pick $D\in|-K_X|_\mathbb{Q}$ such that $(X,D)$ is a klt pair, whose existence is guaranteed by Lemma \ref{lem: Xu_lemma}. So one has $A_X(E)-\ord_E(D)>0$ for all $E$. Using the Zariski decomposition \eqref{eq:Zariski-decom-line-bundle} this implies that
\begin{equation*}
    \begin{aligned}
        A_Z(E)&=A_W(E)+\ord_E(K_W-\pi^*K_X)-\ord_E(r^{-1}F)\\
        &\geq A_W(E)+\ord_E(K_W-\pi^*K_X)-\ord_E(\pi^*D)=A_X(E)-\ord_E(D)>0.
    \end{aligned}
\end{equation*}
So $Z$ has klt singularities, as asserted. Therefore, $Z$ is a $\mathbb{Q}$-Fano variety.

\smallskip
Using the Zariski decomposition \eqref{eq:Zariski-decom-line-bundle} one more time, and ideas from \cite{BBEGZ16,BEGZ10}, we deduce the following result:

\begin{proposition}\label{prop: KE_X_Z}
    The KE metrics in $c_1(-K_X)$ are in one-to-one correspondence with the singular KE metrics in $c_1(-K_Z)$.
\end{proposition}

\begin{proof} We use the terminology underlying formula \eqref{eq:Zariski-decom-line-bundle}.  Let $\omega \in c_1(-K_Z)$ be a smooth background metric on $Z$. Put $\beta:=\tau^*\omega$, which is a smooth semi-positive form in $c_1(-\tau^*K_Z)$. We choose tame measures $\mu_1$ and $\mu_2$ on $X$ and $Z$ respectively in the sense of \cite[Definition 3.1]{BBEGZ16}. More specifically, $\mu_1$ is determined by $\theta\in c_1(-K_X)$ and $\mu_2$ by $\omega\in c_1(-K_Z)$. One can naturally lift two measures to $W$, which will be denoted by $\tilde\mu_1$ and $\tilde\mu_2$ respectively. We further assume that $h$ is a smooth Hermitian metric on $\mathcal{O}_W(F)$ so that $\pi^*\theta=\beta+r^{-1}\ddc\log h$. Let $S_F$ be the canonical section of $\mathcal O_W(F)$.

There exists divisors $E$ and $E'$ on $W$ respectively so that 
$K_W=\tau^*K_Z+E$ and 
$K_W=\pi^*K_X+E^\prime$. Consequently, using \eqref{eq:Zariski-decom-line-bundle} we obtain that $E^\prime=E+\frac{1}{r}F$. Using (the proof of)  \cite[Lemma 3.2(1)]{BBEGZ16}
we obtain that 
$$
\tilde\mu_1=e^{C}(|S_F|^2_h)^{1/r}\tilde\mu_2,
$$
where $C\in\RR$ is some suitable normalization constant.

Now, a KE metric in $c_1(-K_Z)$ corresponds to a bounded potential $u\in PSH(W,\beta)$ such that
$
(\beta+\ddc u)^n=e^{C_1-u}\tilde\mu_2
$
for some constant $C_1\in\RR$.
Let $v:=u+r^{-1}\log|S_F|^2_{h}$, then $v\in PSH(W,\pi^*\theta)$ has minimal singularity type and solves
$$
(\pi^*\theta+\ddc v)^n=e^{C_1-C-\Tilde{u}}\tilde\mu_1,
$$
which gives a KE metric in $c_1(-K_X)$. Conversely, by the same argument, a KE metric in $c_1(-K_X)$ also gives rise to a KE metric in $c_1(-K_Z)$. 
\end{proof}

Thanks to \eqref{eq:Zariski-decom-line-bundle}, it is easy to see that the inequality $\delta_{mr}(-K_Z)\geq\delta_{mr}(-K_X)$ holds as long as one has $\delta_{mr}(-K_X)\geq1$ (here we are using Fujita--Odaka's quantized delta invariant \cite{FO18}). So by \cite[Theorem 4.4]{BJ17} we can let  $m\rightarrow\infty$ to conclude the following result, that is part of \cite[Theorem 1.2]{Xu22}:
\begin{equation}\label{eq: delta_ineq_ample_model}
\delta(Z,-K_Z)\geq\delta(X,-K_X)>1.
\end{equation}

Putting Proposition \ref{prop: KE_X_Z} and \eqref{eq: delta_ineq_ample_model} together it is enough to show that $Z$ has a unique KE metric when $\delta(-K_Z)>1$. This latter condition implies that $Z$ is uniformly K-stable, so $Z$ does not support any non-trivial holomorphic vector field. Theorem \ref{thm:LTW} and \cite[Theorem 11.1]{BBEGZ16} now yield existence of a unique KE metric on $Z$, as desired.

\begingroup
\setstretch{1.1}
\setlength\bibitemsep{0pt}
\setlength\biblabelsep{4pt}
\printbibliography
\endgroup

\small
\noindent {\sc Department of Mathematics, University of Maryland}\\
{\tt tdarvas@umd.edu}\vspace{0.1in}

\noindent {\sc School of Mathematical Sciences, Beijing Normal University}\\
{\tt kwzhang@bnu.edu.cn}
\end{document}